%% file: ckl.tex
\documentclass[reqno]{amsart}

\usepackage {fullpage}						% set all margins to 1.5cm
\usepackage{csquotes}
\usepackage [utf8]{inputenc}
\usepackage {graphicx}
\usepackage [bookmarks=false, pdfstartview={FitH}, colorlinks=true, linkcolor=blue, citecolor=blue, urlcolor=blue]{hyperref} % create the .pdf links
\usepackage {amsfonts,amsmath,amsthm,amssymb}
\usepackage {latexsym}
\usepackage {graphics, graphicx}
\usepackage {fancyhdr}
\usepackage[table]{xcolor}
\usepackage {epstopdf} 						% to automatically transform .eps figs to .pdf
\usepackage {multicol}						% to locally have two columns
\usepackage {multirow}  					% merge rows in a tabular
\usepackage {soul} 							% to break \underline to several lines using \ul
\usepackage {datetime} 						% date and time in latex
\usepackage {subfigure}						% to have enumerated subfigures in one figure
\usepackage {parskip}						% no indentation and space between paragraphs
\usepackage {setspace}						% to set the vertical spacing
\usepackage[font=small,textfont=it,width=0.97\textwidth]{caption}		% to set the fonts for the captions, to use caption*
\usepackage{mathtools}						% to have dcases environment
\usepackage{tikz, pgfplots}
\usetikzlibrary{matrix}
\usepgfplotslibrary{colorbrewer}
\pgfplotsset{cycle list/Dark2}
\pgfplotsset{cycle list/Set2}
\usepgfplotslibrary{groupplots}
\usetikzlibrary{pgfplots.groupplots}
\usetikzlibrary{plotmarks}
\usetikzlibrary{calc}
\pgfplotsset{compat=newest}
\pgfplotsset{plot coordinates/math parser=false}
\pgfplotsset{try min ticks=3}
\newlength\figureheight
\newlength\figurewidth

\newtheorem{lemma}{Lemma}[section]

\theoremstyle{remark}

\numberwithin{equation}{section}

\allowdisplaybreaks[1]

\renewcommand{\(}{\left(}
\renewcommand{\)}{\right)}

\newcommand{\rfr}{^\text{ref}}
\newcommand{\eps}{\varepsilon}
\newcommand{\real}{\mathbb{R}}

\newcommand{\e}{\mathbf e}

\def\softd{{\leavevmode\setbox1=\hbox{d}%
		\hbox to 1.05\wd1{d\kern-0.4ex{\char039}\hss}}}%cstocs

\newcommand{\ddt}{\partial_t}%{\frac{\partial #1}{\partial t}}
\newcommand{\ddx}{\partial_x }%{\frac{\partial #1}{\partial x}}
\newcommand{\ddw}{\partial_w }%{\frac{\partial #1}{\partial x}}
\newcommand\Tstrut{\rule{0pt}{2.6ex}}         % = `top' strut
\newcommand\Bstrut{\rule[-0.9ex]{0pt}{0pt}}   % = `bottom' strut

\title{A hybrid mass transport finite element method\\ for Keller--Segel type systems}

\author{J. A. Carrillo}
\address[J. A. Carrillo]{\newline Department of Mathematics, Imperial College London, London SW7 2AZ, United Kingdom}
\email{carrillo@imperial.ac.uk}

\author{N. Kolbe}
\address[N. Kolbe]{\newline Institute of Mathematics, Johannes Gutenberg-University Mainz}
\email{kolbe@uni-mainz.de}

\author{M. Luk\'a\v{c}ov\'a-Medvi\v{d}ov\'a}
\address[M. Luk\'a\v{c}ov\'a-Medvid'ov\'a]{\newline Institute of Mathematics, Johannes Gutenberg-University Mainz}
\email{lukacova@uni-mainz.de}

\begin{document}

\maketitle

\vspace{-1cm}

\begin{abstract}
We propose a new splitting scheme for general reaction-taxis-diffusion systems in one spatial dimension capable to deal with simultaneous concentrated and diffusive regions as well as travelling waves and merging phenomena. The splitting scheme is based on a mass transport strategy for the cell density coupled with classical finite element approximations for the rest of the system. The built-in mass adaption of the scheme allows for an excellent performance even with respect to dedicated mesh-adapted AMR schemes in original variables.
\end{abstract}

\medskip

{\bf Keywords:} mass transport schemes, reaction-aggregation-diffusion systems, splitting schemes, tumor invasion models

\medskip

\section{Introduction}
The aim of the present work is to design a numerical scheme capable to deal with concentrations and diffusion phenomena typically arising in one-dimensional taxis-diffusion systems of the form
\begin{equation} \label{eq:systemtype}\left\{
\begin{aligned}
\ddt \rho &= \ddx  \(D_\rho \ddx \rho - \chi \rho \ddx c \) + R_\rho(\rho)& \text{in }(0,\infty) \times (a,b), \\
\varepsilon \ddt c &= D_c \ddx^2 c + R_c(\rho, c) & \text{in }(0,\infty) \times (a,b),\\
	\ddx \rho(\cdot, r) &= \ddx c(\cdot, r)=0, &r\in\{a,b\},\\
	\rho(0,\cdot) &= \rho_0 \geq 0, \quad  c(0,\cdot) = c_0 \geq 0 &
\end{aligned}\right.
\end{equation}
with Lipschitz continuous source terms $R_\rho, \, R_c$ that satisfy $R_\rho(0), R_c(\rho, 0)\geq 0$. Here $\rho$ denotes the cell density and $c$ the concentration of a chemo-attractant. These systems constitute adaptations of the classical cell migration model by Patlak, Keller and Segel \cite{Patlak.1953, KS.1970}. They have been widely used in the modeling of biological processes such as cell organization in tissue, immune system dynamics and cancer growth \cite{Esipov.1998, Anderson.2000, EMT_paper}.
The dynamics of their solutions are quite rich; apart from traveling waves \cite{Kong.2008} the aggregation phenomenon studied in \cite{jager1992explosions,BDP} that leads to blowup in finite time is of specific interest. One has moreover observed the occurrence of high concentrations that can emerge in a smooth solution, split, and merge with each other \cite{Painter.2011}. Nonlinear diffusions or saturated responses in the chemotactic sensitivity are natural ways to include volume filling effects into the models, see \cite{HPVolume,CCVolume}. They usually avoid blow-up in a biologically meaningful way and lead to interesting phenomena and asymptotic stabilization. Finally, these models are basic building bricks for a variety of cancer invasion models in the literature \cite{Chaplain.2005, EMT_paper,  Preziosi.2003, Stinner.2015, Johnston.2010} in which the coupling with extracellular matrix, enzymatic activators and other substances are taken into account. One of the common features in all of these models is the simultaneous occurrence of regions of high concentrated densities with diffuse profiles leading to numerical difficulties in choosing well-adapted meshes.
The numerical approximation of all of these simultaneous phenomena is particularly challenging.

In \cite{Carillo.2008} a mass transport steepest descent scheme has been proposed to resolve a modified 1D Keller-Segel system for the log interaction kernel proposed in \cite{Calvez.2007}. The method satisfies a discrete free energy dissipation principle by design being based on the variational schemes for Fokker-Planck type equations introduced in \cite{jordan1998variational,KW} and applied to Keller-Segel type models in \cite{Carillo.2008,Blan.sys}. By considering the problem in transformed variables the method can resolve areas of high concentrations accurately without any mesh refinement. This approach has been extended to several dimensions for nonlinear aggregation-diffusion equations and with different approaches in the discretization in \cite{carrillo2009numerical,MO2014,CRW,JMO} and the references therein.

The aim of this work is to extend the mass transport approach to the general class of systems \eqref{eq:systemtype}. We will test different scenarios that feature in particular the splitting, traveling and emerging of concentrations. For the adjustment of the scheme we propose a splitting method, where we employ the technique from \cite{Carillo.2008} to the Keller-Segel part of the system (i.e. the first equation of \eqref{eq:systemtype} with $R_c = 0$). The remaining system of an ODE and a diffusion reaction equation will be decoupled and solved by a suitable finite element method. The advantage of the mass transport approach for the cell densities equations is that the mesh adapts naturally to the mass distribution, and then coarse meshes in the mass variable can still lead to good numerical approximations as we will discuss below.

In more details, we split \eqref{eq:systemtype} into two subsystems. The solution of the full system \eqref{eq:systemtype} can then be approximated by appropriately combining short time solution of the subsystems. We introduce at first the diffusion-advection system given by
\begin{equation} \label{eq:split_KS} \tag{I}\left\{
\begin{aligned}
\ddt \rho &= \ddx \(D_\rho \ddx \rho - \chi \rho \ddx c \) & \text{in }(0,\infty) \times (a,b), \\
\ddt c &= 0 & \text{in }(0,\infty) \times (a,b), \\
	\ddx \rho(\cdot, r) &=0, &r\in\{a,b\},\\
\rho(0,\cdot) &= \rho^I_0 \geq 0, \quad  c(0,\cdot) = c^I_0 \geq 0.
\end{aligned}\right.
\end{equation}
This system makes the assumption of a steady chemo-attractant density $c$ and mass conservation in the cell density $\rho$. Second, we consider the reaction-diffusion system
\begin{equation} \label{eq:split_ReacDiff} \tag{II}\left\{
\begin{aligned}
\ddt \rho&=  R_\rho(\rho )& \text{in }(0,\infty) \times (a,b),  \\
\ddt c &= D_c \ddx^2 c + R_c(\rho, c)& \text{in }(0,\infty) \times (a,b),\\
\ddx c(\cdot, r)&=0, &r\in\{a,b\}, \\
\rho(0,\cdot) &= \rho^{II}_0 \geq 0, \quad  c(0,\cdot) = c^{II}_0\geq 0&
\end{aligned}\right.
\end{equation}
that contains the remaining terms of the system. Following the mass transport algorithm \cite{Carillo.2008} we transform the system \eqref{eq:split_KS} into new variables. With this aim, we consider the pseudo inverse cumulative distribution of the cell density $\rho$,
\begin{equation}\label{eq:pseudoinv}
	V(t, w) = \inf \left\{ y : \int_{a_I}^y c(x,t)\, dx > w \right\},
\end{equation}
which is defined by
$$
0\leq w \leq \int_a^b \rho(t,x)\, dx = m(t) \,.
$$
The system \eqref{eq:split_KS} can now be rewritten, following e.g. \cite{Carrillo.2005}, as
\begin{equation} \label{eq:split_KS_trans} \tag{I'}\left\{
\begin{aligned}
\ddt V &= - D_\rho \ddw \( \left[ \ddw V \right]^{-1} \) + \chi \ddx c |_{(x=V(w))}& \text{in }(0,\infty) \times (0,m),\\
\ddt c &= 0 &  \text{in }(0,\infty) \times (a,b),\\
V(\cdot, 0) &= a, \qquad V(\cdot, M) = b,\\
V(0,\cdot) &= V_I , \quad  c(0,\cdot) = c_I \geq 0, &
\end{aligned}\right.
\end{equation}
where $m$ denotes a given mass during this splitting step. The advantage of the proposed splitting is that the mass of cell densities does not change over the first step and the cell density is fixed over the second step.

The details of the full discretization of the proposed splitting scheme will be given in Section 2. In Section 3 we discuss the choice of the constraints in the time, spatial and mass steppings due to the choice of the full discrete schemes. Section 4 is devoted to study in detail the performance of this splitting scheme in many complex situations ranging from the simpler Keller-Segel type systems and their small variations to quite more biologically relevant systems in tumor invasion as discussed above. We will analyze the experimental convergence and the computational cost of this discretization with respect to previous schemes with mesh-refinement algorithms in original spatial variables. Finally we conclude in Section 5.

%%%%%%%%%%%%%%%%%%%%%%%%%%%%%%%%%%%

\section{Numerical method} \label{section:num}
In what follows, we describe a numerical treatment for both systems \eqref{eq:split_KS_trans} and \eqref{eq:split_ReacDiff}. The inverse distribution $V$ is given on the time evolving \emph{mass space} $(0,m_h(t))$, whereas the chemo-attractant $c$ is given in the Eulerian coordinates in $(a,b)$. This leads to two meshes that the proposed numerical method employs.

First, we discretize the normalized mass domain $(0,1)$, on which the pseudo inverse distribution $V$ resides by the mesh
\begin{equation*}
	0 = w_0 < w_1 <\dots <w_M=1, \quad w_j = j  h_w, \quad j = 0,\dots, M
\end{equation*}
with length $M \in \mathbb{N}$ and width $h_w = 1/M$ that corresponds to the width $\Delta w(t)=m_h(t) h_w$ in the time evolving mass domain $(0,m_h(t))$. We denote the point values of $V$ by $V_j (t) = V(m_h(t)\,w_j, t)$ for $j=0,\dots,M$ and introduce the linear spline in $w$ connecting the discrete values that we denote by $V_h(t,m_h(t)\,w)$. Here we have used the discrete mass of the cells
\begin{equation*}
m_h(t) = \int_a^b \rho_h(t,x)\, dx,
\end{equation*}
where $\rho_h$ is a discrete representation of the cell density to be defined later on.

A second mesh partitions the physical space $(a,b)$ for the chemo-attractant density $c$ into
 \begin{equation}\label{eq:c_grid}
 a = x_0 < x_1 <\dots <x_N=b, \quad x_k = a + k  \Delta x, \quad k = 0,\dots, N.
 \end{equation}
 The chemo-attractant mesh is thus of length $N$ and width $\Delta x = (b-a)/N$.
We employ a linear finite element representation for the chemo-attractant density $c$. Therefore let $\{ \phi_k, ~k=1,\dots, N-2\}$ be the basis of piecewise linear hat functions on the grid \eqref{eq:c_grid} satisfying the boundary conditions. In particular, we have
\begin{equation*}
\phi_k(x) = \begin{cases}
(x-x_{k-1}) /\Delta x,& x_{k-1} \leq x \leq x_k, \\
(x_{k+1}- x) /\Delta x,&  x_{k} \leq x \leq x_{k+1},\\
0, &\text{otherwise}
\end{cases}, \qquad k= 2, \dots, N-2.
\end{equation*}
in the center of the domain and
\begin{align*}
\phi_1(x) &= \begin{cases}
1 ,& a \leq x \leq x_1, \\
(x_2  - x) /\Delta x,&  x_{1} \leq x \leq x_{2},\\
0, &\text{otherwise},
\end{cases}\\[1ex]
\phi_{N-1}(x) &= \begin{cases}
(x-x_{N-2} ) /\Delta x,& x_{N-2} \leq x \leq x_{N-1}, \\
1,&  x_{N-1} \leq x \leq b,\\
0, &\text{otherwise}
\end{cases}
\end{align*}
near the boundary. By using the basis functions we can define the approximate chemo-attractant density as
\begin{equation*}
c_h(x,t) = \sum_{k=1}^N c_i(t) \phi_i(x).
\end{equation*}
For the construction of the splitting method we define solution operators for both systems \eqref{eq:split_KS_trans} and \eqref{eq:split_ReacDiff}.
 To this end we design $T$ to be a numerical solution operator of system \eqref{eq:split_KS_trans} in the following sense: if $(V_h(\tilde t), c_h(\tilde t), m_h(\tilde t))$ is a numerical solution at $t=\tilde t$ then $ T_{\Delta t}(V_h(\tilde t), c_h(\tilde t), m_h(\tilde t)) $ is a numerical solution of system \eqref{eq:split_KS_trans} at time $t = \tilde t + \Delta t$. In the same manner, we define also a solution operator $S$ for system \eqref{eq:split_ReacDiff}.

\subsection{The solution operator $T$ for system \eqref{eq:split_KS_trans}}

For a discretization of the system $\eqref{eq:split_KS_trans}$ we need to evaluate the derivative of the chemo-attractant concentration in the state variable $V$. With this aim we consider an interpolation by cubic splines of the discrete chemo-attractant concentration.
Let $(V_h(t), c_h(t), m_h(t))$ be given initial data. By $\hat c_h$ we denote the cubic spline over the data points $(x_k, c_h(t,x_k))$ for $k=1,\dots,N$ that satisfies the boundary conditions $\ddx \hat c_h(a) = \ddx \hat c_h(b)=0$. We use this spline for the approximation of the advection term. Concerning the time integration we split the taxis and diffusion terms and treat the stiff diffusion terms implicitly. In this way we allow for both large time steps and stability of the scheme. We apply in particular the two stage implicit-explicit midpoint scheme (see e.g. \cite{pareschi2005implicit}) that reads in our case
\begin{subequations}
\begin{equation}\label{eq:V_scheme_I}
- 2\frac{\tilde V_j (t) - V_j(t)}{\Delta t} = \frac{D_\rho}{\tilde V_{j+1}(t) - \tilde V_{j}(t)} - \frac{D_\rho}{\tilde V_j(t) - \tilde V_{j-1}(t)}
- \chi  \ddx \hat c_h(V_j(t)),
\end{equation}

\begin{equation}\label{eq:V_scheme_II}
- \frac{T_{\Delta t} V_j (t) - V_j(t)}{\Delta t} = \frac{D_\rho}{\tilde V_{j+1}(t) - \tilde V_{j}(t)} - \frac{D_\rho}{\tilde V_j(t) - \tilde V_{j-1}(t)}
- \chi \ddx \hat c_h(\tilde V_j(t))
\end{equation}
\end{subequations}
both for $j=0, \dots,M$. We have approximated the diffusion terms above by a central difference formula as in \cite{Carillo.2008}. At the boundary we impose Neumann boundary
conditions, i.e.
$$
\frac{1}{\tilde V_{M+1}(t) -\tilde V_{M}(t)} = \frac{1}{\tilde V_{0}(t) -\tilde V_{-1}(t)}=0\,.
$$
The intermediate stage $\tilde V_j(t)$ is given by a nonlinear implicit equation \eqref{eq:V_scheme_I} and we use the Newton's method for its computation. For the computation of the taxis terms in \eqref{eq:V_scheme_I} and \eqref{eq:V_scheme_II} we evaluate the afore determined spline $\hat c_h$.

The chemo-attractant density as well as the mass of the cells are not affected by system $\eqref{eq:split_KS_trans}$, hence we define the numerical operator accordingly by
$$
T_{\Delta t} c_h(t) = c_h(t), \quad T_{\Delta t} m_h(t) = m_h(t).
$$

Note that if instead of linear diffusion, i.e. $D_\rho$ constant, we have a power-law nonlinear diffusion $D_\rho(\rho)=D_\rho \rho^{\gamma-1}$, $\gamma>1$, modelling cell volume size effects as in \cite{HPVolume,CCVolume}, we obtain a similar approximation
\begin{subequations}
	\begin{equation}\label{eq:V_scheme_IKSLnD}
	- 2\frac{\tilde V_j (t) - V_j(t)}{\Delta t} = \frac{\tilde D(t)}{(\tilde V_{j+1}(t) - \tilde V_{j}(t))^{\gamma}} - \frac{\tilde D(t)}{(\tilde V_j(t) - \tilde V_{j-1}(t))^{\gamma}} - \chi  \ddx \hat c_h(V_j(t)),	\end{equation}
	
	\begin{equation}\label{eq:V_scheme_IIKSLnD}
	- \frac{T_{\Delta t} V_j (t) - V_j(t)}{\Delta t} = \frac{\tilde D(t)}{(\tilde V_{j+1}(t) - \tilde V_{j}(t))^{\gamma}} - \frac{\tilde D(t)}{(\tilde V_j(t) - \tilde V_{j-1}(t))^{\gamma}}
	- \chi  \ddx \hat c_h(\tilde V_j(t))
	\end{equation}
\end{subequations}
with $\tilde D(t)=D_\rho\gamma^{-1}\Delta w(t)^{\gamma-1}$, $j=0,\dots, M$ and similar boundary conditions as above. Remember that the continuous function $T_{\Delta t} V_h(t)$ is built as the linear interpolant of the values $T_{\Delta t} V_{j}(t)$ for $j=0,\dots, M$, and thus we can define a reconstructed density $T_{\Delta t}\rho_h (t)$ by its own definition
\begin{equation}\label{eq:recden}
T_{\Delta t}\rho_h (t) = \left( \frac{\partial T_{\Delta t}V_h(t)}{\partial w}\right)^{-1}
\end{equation}
as long as the sequence $V_{j}(t)$ is strictly increasing.

\subsection{The solution operator $S$ for system \eqref{eq:split_ReacDiff}}

In the splitting method that we propose we will apply the reaction-diffusion operator $S$ starting with the data $(T_{\Delta t} V_h(t), T_{\Delta t}c_h( t), T_{\Delta t}m_h(t))$ obtained from a previous evaluation of the operator $T$. For simplicity we will describe the numerical operator $S$ for general initial data $(V_h(t), c_h( t), m_h(t))$.

System \eqref{eq:split_ReacDiff} is formulated for physical concentrations of cells. To provide adequate initial data using the given approximations $(V_h(t), c_h(t), m_h(t))$ we transform the discrete pseudo inverse distribution $V_h(t)$ on $(0,m_h(t))$ to a finite volume representation of $\rho(t, \cdot)$ on $(a,b)$. Since the approximate density $\rho_h$ satisfies
\begin{equation*}
\int_{V_{j-1}(t)}^{V_j(t)} \rho_h(t,x) \, dx = \Delta w(t),
\end{equation*}
for all $j=1,\dots,M$ by construction \eqref{eq:recden}, we can introduce the cell averages and the piecewise constant function $\rho_h$ in the following way
\begin{equation*}%\label{eq:def_rhoh}
\rho_j(t) = \frac{\Delta w(t)}{V_j(t)- V_{j-1}(t)}, \quad j = 1,\dots,M, \quad \rho_h(t,x) = \sum_{j=1}^{M} \rho_j(t)\chi_{(V_{j-1}(t), V_j(t))}(x).
\end{equation*}
This approximation of the cell density resides on physical space $(a,b)$. Note though that the cell averages are given on a non-uniform grid which differs from the grid for the chemo-attractant density $c$ given in \eqref{eq:c_grid}.

Now, we are in the position to write down the scheme for system \eqref{eq:split_ReacDiff}. Again we split diffusion from reaction and apply the implicit-explicit midpoint scheme and obtain
\begin{subequations}\label{eq:reac_update}
\begin{equation}\label{eq:cell_reac_update1}
\tilde \rho_j(t) = \rho_j(t) + \frac{\Delta t}{2} R_\rho (\rho_j(t)), \quad j=1, \dots,M,
\end{equation}
\begin{equation}\label{eq:attr_update1}
2 \varepsilon \,\frac{\tilde c_k(t) - c_k(t) }{\Delta t}\!\int_{a}^{b} \!\!\phi_k \phi_l \,dx = -\tilde c_k(t) D_c \int_{a}^{b} \frac{\partial \phi_k}{\partial x} \frac{\partial \phi_l}{\partial x}\,dx + \int_{a}^{b} \!\!R_c(\rho_h(t), c_h(t)) \phi_l\, dx, \quad k,l = 1,\dots, N-1,
\end{equation}
\begin{equation} \label{eq:cell_reac_update2}
S_{\Delta t}\rho_j(t) = \rho_j(t) + \frac{\Delta t}{2} R_\rho (\tilde \rho_j(t)), \quad j=1, \dots,M,
\end{equation}
\begin{equation}\label{eq:attr_update2}
\varepsilon \,\frac{S_{\Delta t} c_k(t) - c_k(t) }{\Delta t}\!\!\int_{a}^{b} \!\!\phi_k \phi_l \,dx = -\tilde c_k(t) D_c \!\! \int_{a}^{b} \!\frac{\partial \phi_k}{\partial x} \frac{\partial \phi_l}{\partial x}\,dx + \!\!\int_{a}^{b} \!\!\!R_c(\tilde \rho_h(t), \tilde c_h(t)) \phi_l\, dx, \quad k,l = 1,\dots, N-1.
\end{equation}
\end{subequations}
As usual, we employ precomputed integrals of the basis functions
$$
 \int_{a}^{b} \phi_k \phi_l \,dx  \qquad \mbox{and} \qquad
\int_{a}^{b} \frac{\partial \phi_k}{\partial x} \frac{\partial \phi_l}{\partial x}\,dx
$$
in the computation of the linear systems \eqref{eq:attr_update1} and \eqref{eq:attr_update2}. The integrals of the form
 $\int_{a}^{b} R_c(\rho_h(t), c_h(t)) \phi_l\, dx$ are dependent on $V_h(t)$. For their computation we use suitable quadratures together with an indicator function to identify the position of a particular point $x\in [a,b]$ on the grid corresponding to the cell density $\rho_h$.
The reaction update in the cell density $c_h$ alters the mass of the cells over the interval $\Omega$. Thus we update $m_h(t)$ by
\begin{equation*}
S_{\Delta t} m_h(t) = \sum_{j=1}^{M} S_{\Delta t}\rho_j(t) (V_j(t)- V_{j-1}(t)).
\end{equation*}
To be able to apply the advection-diffusion operator after the reaction-diffusion update we transform $S_{\Delta t}\rho_h(t)$ to its inverse distribution representation $S_{\Delta t} V_j(t)$. Therefore, we use the formula
\begin{equation}\label{eq:V_reaction_update}
\int_{S_{\Delta t} V_{j-1}(t)}^{S_{\Delta t} V_j(t)} \sum_{j=1}^{M} S_{\Delta t} \rho_j(t)\, \chi_{(V_{j-1}(t), V_j(t))}(x)\, dx = S_{\Delta t} m_h(t) h_w, \quad j = 1,\dots,M.
\end{equation}
As long as $S_{\Delta t} V_j(t)$ is monotonically increasing in $j$, identity \eqref{eq:V_reaction_update} allows for an efficient update of the inverse distribution $V_h$.

\subsection{The splitting method}
To approximate the full system \eqref{eq:systemtype} we propose the classical Strang splitting method \cite{strang} employing both numerical operators defined above.
For given non-negative and sufficiently smooth initial conditions $\rho_0$ and $c_0$ of system \eqref{eq:systemtype} we deduce discrete initial data $(V_h(0), c_h(0), m_h(0))$. To compute a discrete representation $V_h(0)$ of the normalized concentration $\rho_0/m_h(0)$ we integrate as in \eqref{eq:V_reaction_update}.

Then we define the fully discrete Strang splitting scheme for system \eqref{eq:systemtype} iteratively by
\begin{equation}\label{eq:full_strang}
(V_h(t^{n+1}), c_h( t^{n+1}), m_h(t^{n+1})) = T_{\Delta t^n /2} S_{\Delta t^n} T_{\Delta t^n /2} (V_h(t^n), c_h( t^n), m_h(t^n)), \quad n = 0,1,2, \dots ,
\end{equation}
where $0=t^0< t^n = \sum_{i=1}^{n} \Delta t^i$ is a discretization of the time axis.
In this way we alternate between applying the diffusion-taxis and the diffusion-reaction operator. The symmetrical structure leads to the second order splitting error.

To optimize the efficiency we adapt the time increment $\Delta t$ in each time step. Since the discretization of system \eqref{eq:split_KS} is more sensitive to instabilities that are caused by large time increments $\Delta t$ than the discretization of the diffusion--reaction system, we start the method in each time step with the numerical operator $T$ in which we determine $\Delta t^n$. We will elaborate on the stability of the scheme in Section \ref{sec:monotonicity}.

The scheme \eqref{eq:full_strang} is not limited to the case of a single pair of a cell and an chemo-attractant. An extension to multiple attractants (i.e. a replacement of $\chi \rho \nabla c$ by a sum $\chi_1 \rho \nabla c_1 + \dots + \chi_n \rho \nabla c_n$ in \eqref{eq:systemtype}) is straightforward. The case of multiple cell densities coupled through the taxis terms, such as in the model discussed in \cite{EMT_paper}, can be treated as well. Note though that each cell species brings along another non-uniform mesh on the domain $(a,b)$ which requires further projections in the numerical operator $S$.

%%%%%%%%%%%%%%%%%%%%%%%%%%%%%%%%%%%%%%%%%

\section{Monotonicity preservation of the diffusion-taxis operator}\label{sec:monotonicity}
As demonstrated in \cite{Chertock.2008} unphysical negative values that arise in the numerical solutions of the Keller-Segel type systems can cause instabilities and wrong behavior of the scheme. Hence, the so called \emph{positivity preserving} finite volumes schemes for these kind of models have been developed, e.g. in \cite{Chertock.2008}. A non-negative density $\rho$ implies a monotonously increasing pseudo inverse distribution $V$ by its definition \eqref{eq:pseudoinv}. If a method operates on inverse distributions it should in turn preserve the discrete monotonicity of $V$. This \emph{monotonicity preserving} property of such schemes was studied in the case of filtration and convolution-diffusion equations in \cite{Gosse.2006.Lagrangian, Gosse.2006.Filtration}. In more details, We call a method monotonicity preserving if from $V_j(t)-V_{j-1}(t) > 0$ for all $0<j\leq M$ follows that also $V_j(t+\Delta t)-V_{j-1}(t + \Delta t) > 0$ for all $0<j\leq M$. 

In the rest of this section we focus on a simplified problem that motivates a way to adapt the time increment $\Delta t$ in such a way, that non-monotone solutions and thus possible related instabilities are avoided. We consider in particular the system \eqref{eq:split_KS_trans} for the case of a steady chemo-attractant $c \in C^1(a,b)$.
For the numerical resolution we consider a forward Euler scheme of the form
\begin{equation}\label{eq:forward_simple_diff}
V_j (t+\Delta t)  = V_j(t) + \Delta t\, \chi \ddx c(V_j(t)) -\Delta t \left[\frac{\tilde D(t)}{( V_{j+1}(t) - V_{j}(t))^{\gamma}} - \frac{\tilde D(t)}{(V_j(t) -  V_{j-1}(t))^{\gamma}} \right],
\end{equation}
for a discrete inverse distribution as defined in Section \ref{section:num}. This scheme can be understood as an explicit first-order version of the advection-diffusion operator introduced in the previous section.
In this setting we can follow the lines of \cite{Gosse.2006.Lagrangian, Gosse.2006.Filtration} and derive a bound on $\Delta t$ that makes the scheme \eqref{eq:forward_simple_diff} monotonicity preserving:

\begin{lemma}\label{lem:cfl}
The scheme \eqref{eq:forward_simple_diff} is monotonicity preserving, if for a fixed $\theta \in(0,1)$ both CFL conditions
\begin{subequations}
\begin{align}
\Delta t ^n &< \frac{\theta}{2\, D_\rho\, \Delta w ^{\gamma-1}} \,\min_{0 \leq j < M}  \frac{ (V_{j+1}(t^n) - V_j(t^n)) (V_{j}(t^n) - V_{j-1}(t^n))}{\max_{k=j-1,j} ~(V_{k+1}(t^n) - V_k(t^n))^{-(\gamma -1)} }, \label{eq:cfl_diff}\\[1ex]
\Delta t ^n &<  \frac{1-\theta}{\chi} \, \min_{0 \leq j < M}  \frac{ (V_{j+1}(t^n) - V_j(t^n)) }{ \left| \ddx c(V_{j+1}(t^n)) - \ddx c(V_{j}(t^n)) \right|}\label{eq:cfl_chemo}
\end{align}
\end{subequations}
 are satisfied.
\end{lemma}
\begin{proof}
		We consider a single time step in the scheme \eqref{eq:forward_simple_diff} and drop the superscript $n$. For brevity we will use the notation $\Delta V_{j+1/2} = V_{j+1}- V_j$. We assume the monotonicity of the discrete inverse distribution at the time instance $t$ and compute for an arbitrary $0\leq j<M$ the difference
	\begin{align*}%\label{eq:vjp1_minus_vj}
	\Delta V_{j+1/2}(t + \Delta t) &= \Delta V_{j+1/2}(t) + \Delta t\, \chi \(  \ddx c(V_{j+1}(t)) - \ddx c(V_j(t))\) \nonumber \\[1ex]
	&\quad - \frac{\Delta t \,D_\rho}{\Delta w} \left[\frac{(\Delta w) ^\gamma}{\gamma \, ( \Delta V_{j+3/2}(t))^{\gamma}} - \frac{(\Delta w) ^\gamma}{\gamma \, ( \Delta V_{j+1/2}(t))^{\gamma}}   \right] \nonumber \\&\quad+ \frac{\Delta t \,D_\rho}{\Delta w} \left[\frac{(\Delta w) ^\gamma}{\gamma \,(\Delta V_{j+1/2}(t))^{\gamma}} - \frac{(\Delta w) ^\gamma}{\gamma \,( \Delta V_{j-1/2}(t))^{\gamma}}   \right].
	\end{align*}
	By applying the mean value theorem to the function $f(x) = x^{\gamma}/\gamma$ we find two function evaluations of its derivative, $\kappa_{j}$ and $\kappa_{j+1}$, such that we obtain
	\begin{align*}
	\Delta V_{j+1/2}(t + \Delta t) &= \Delta V_{j+1/2}(t) + \Delta t\, \chi \(  \ddx c(V_{j+1}(t)) - \ddx c(V_j(t))\)  \\[1ex]
	&\quad - \Delta t \,D_\rho\, \kappa_{j+1} \left[\frac{1}{\Delta V_{j+3/2}(t)} - \frac{1}{\Delta V_{j+1/2}(t)}   \right]  \quad+ \Delta t \,D_\rho \,\kappa_{j}\left[\frac{1}{\Delta V_{j+1/2}(t)} - \frac{1}{\Delta V_{j-1/2}(t)}   \right].
	\end{align*}
	Note that by the non-negativity of $\Delta V_{j+1/2}$ both $\kappa_{j}$ and $\kappa_{j+1}$ are non-negative.
	In the next step, we define $L_{j+1/2}= (\ddx c(V_{j+1}(t)) - \ddx c(V_j(t)))/(V_{j+1}(t) - V_j(t))$ and rewrite
	\begin{align*}
	\Delta V_{j+1/2}(t + \Delta t) &=\Delta V_{j+1/2}(t) \( 1 +  \Delta t\, \chi L_{j+1/2} - \frac{\Delta t \,D_\rho\, \kappa_{j+1}}{\Delta V_{j+3/2}(t) \, \Delta V_{j+1/2}(t)} - \frac{\Delta t \,D_\rho\, \kappa_{j}}{\Delta V_{j+1/2}(t) \Delta V_{j-1/2}(t)}\) \\[1ex]
	&\quad + \frac{\Delta t \,D_\rho\, \kappa_{j+1}}{\Delta V_{j+3/2}(t) \, \Delta V_{j+1/2}(t)}\, \Delta V_{j+3/2}(t)
	+  \frac{\Delta t \,D_\rho\, \kappa_{j}}{\Delta V_{j+1/2}(t) \Delta V_{j-1/2}(t)}  \,  \Delta V_{j-1/2}(t).%- \Delta t \,D_\rho\,
	\end{align*}
	Finally we estimate by the monotonicity at time instance $t$
	\begin{equation}\label{eq:dv_est}
		\Delta V_{j+1/2}(t + \Delta t) \geq \Delta V_{j+1/2}(t) \( 1 -  \Delta t\, \chi |L_{j+1/2}| \!- \!\frac{\Delta t \,D_\rho\, \kappa_{j+1}}{\Delta V_{j+3/2}(t) \, \Delta V_{j+1/2}(t)}\! -\! \frac{\Delta t \,D_\rho\, \kappa_{j}}{\Delta V_{j+1/2}(t) \Delta V_{j-1/2}(t)}\)\!.
	\end{equation}
	By using the conditions \eqref{eq:cfl_diff} and \eqref{eq:cfl_chemo}, the non-negativity of the right hand side in \eqref{eq:dv_est} follows. This implies the monotonicity-preserving property of the scheme \eqref{eq:forward_simple_diff}.
\end{proof}

For our splitting method \eqref{eq:full_strang} we assume that we avoid time step restrictions due to the diffusion terms by our implicit treatment and take a closer look on the condition \eqref{eq:cfl_chemo} ($\theta = 0$).
The point values of the inverse distribution $V_j(t)$ for $0\leq j \leq M$ coincide with the mesh cell interfaces of the non-uniform mesh corresponding to the cell densities $\rho_h(t)$. Thus the quantity $L_{j+1/2}$ in the proof of Lemma \ref{lem:cfl} can be understood as a finite difference formula for the second derivative of the chemo-attractant density $c$. In effect, the above CFL condition \eqref{eq:cfl_chemo} motivates to choose the time increment $\Delta t^n$ according to
\begin{equation}\label{eq:splschemecfl}
\Delta t^n \propto  \( \chi  \sup_{\{ x \in I\}}|\ddx^2 c(x)| \)^{-1}.
\end{equation}

For our numerical experiments with the more complex scheme \eqref{eq:full_strang} we have accordingly computed the time increments by
\begin{equation}\label{eq:gf_cfl}
\Delta t ^n = CFL ~ \min \left\{\min_{0 \leq j < M}  \frac{ (V_{j+1}(t^n) - V_j(t^n)) }{\chi \, \left| \ddx c(V_{j+1}(t^n)) - \ddx c(V_{j}(t^n))\right|}, ~ K \Delta w \right\}
\end{equation}
for constants $CFL, K>0$. The additional bound proportional to $\Delta w$ balances the temporal and the spatial errors; large values of $K$ can be used in practice. We have chosen $CFL = 0.49$ and $K=100$ in our numerical experiments. Using this condition we have not observed any non-monotone numerical solutions in our experiments and no instabilities have occurred.

%%%%%%%%%%%%%%%%%%%%%%%%%%%%%%%%%%%%%%%%

\section{Numerical experiments}\label{sec:experiments}

In this section we apply our newly developed mass transport method to several models arising in biomedical applications that bring along merging, emerging, and traveling concentrations phenomena. In particular, we consider the  classical Keller-Segel model both elliptic and parabolic. We study also a simple as well as a detailed cancer invasion model. The latter takes the role of the serine protease urokinase-type plasminogen activator into account. The numerical study of such systems constitutes a challenge due to the complex behavior that the solutions exhibit.
Numerical experiments presented below demonstrate the robustness and reliability of our newly developed mass transport finite element method.

\subsection{A parabolic-elliptic Keller-Segel model with logistic growth}
In the first test case we consider the modified KS system from \cite{Calvez.2007} with added logistic growth which reads
\begin{equation}\label{eq:PKS_lgg}\left\{
\begin{aligned}
\ddt \rho &= \ddx \( \ddx \rho - \chi  \rho \ddx c\) + \mu \rho (1- \rho)& \text{in }(0,\infty) \times \real, \\
c(\cdot, x) &= - \frac{1}{\pi} \int_\real \log(|x-y|)\rho(\cdot,y)\,dy &\text{in }(0,\infty) \times \real,\\
\rho(0,\cdot) &= \rho_0 \geq0.
\end{aligned}\right.
\end{equation}
Note that the adaptation of system \eqref{eq:PKS_gf} to $2D$ with $\mu =0$ is equivalent to the simplified Keller-Segel system from \cite{jager1992explosions}, where the chemo-attractant $c$ is determined by a Poisson equation. The logistic term accounts for additional cell growth that is locally limited by resources and space. Global existence of solutions to the parabolic-parabolic model with logistic growth in 2D was shown in \cite{osaki2002exponential}. Except for the logistic source term this model has been numerically investigated by the mass transport scheme employing only inverse distributions in \cite{Carillo.2008}.

Since the chemo-attractant density $c$ is given by a convolution term, we do not need to use a finite element approximation. Instead we proceed as in \cite{Carillo.2008} and expand the diffusion taxis operator by
\begin{subequations}\label{eq:V_scheme_IKS}
	\begin{equation}\label{eq:V_scheme_IKSL}
	- 2\frac{\tilde V_j (t) - V_j(t)}{\Delta t} = \frac{1}{\tilde V_{j+1}(t) - \tilde V_{j}(t)} - \frac{1}{\tilde V_j(t) - \tilde V_{j-1}(t)} +\frac{\chi\, \Delta w}{ \pi} \lim\limits_{\varepsilon \rightarrow 0} \sum_{i: |\tilde V_j(t) - \tilde V_i(t)|\geq \varepsilon} \frac{1}{\tilde V_j(t)- \tilde V_i(t)},	\end{equation}
	
	\begin{equation}\label{eq:V_scheme_IIKSL}
	- \frac{T_{\Delta t} V_j (t) - V_j(t)}{\Delta t} = \frac{1}{\tilde V_{j+1}(t) - \tilde V_{j}(t)} - \frac{1}{\tilde V_j(t) - \tilde V_{j-1}(t)}
	 +\frac{\chi\, \Delta w}{ \pi} \lim\limits_{\varepsilon \rightarrow 0} \sum_{i: |\tilde V_j(t) - \tilde V_i(t)|\geq \varepsilon} \frac{1}{\tilde V_j(t)- \tilde V_i(t)}.
	\end{equation}
\end{subequations}
During the computation of \eqref{eq:V_scheme_IKSL} we control the convergence of the Newton method by comparing subsequent iterates. If the iteration fails to converge, we abort the computation assuming blowup of the numerical solution. The time increments in these experiments have been adapted such, that the Jacobian of \eqref{eq:V_scheme_IKSL} that occurs in the Newton iteration is strictly diagonally dominant.
The second operator $S$ in this setting accounts only for the logistic growth term. For the numerical simulations we use a grid with only $M=50$ points.

We consider a numerical experiment with the parameters $D_\rho = 1,~ \chi =2.5 \pi$ and the initial datum given by
\begin{equation}\label{eq:V_singlePeak_ini}
V_0(w) = \frac{w- 0.5}{\sqrt[4]{(w + 0.01)\, (1.01 - w)}}.
\end{equation}
This experiment has been studied in the case $\mu =0$ in \cite{Carillo.2008}, where blowup in final time around $t=0.33$ has been obtained numerically. We confirm the same phenomenon using the splitting method, see Figure \ref{fig:exp431}. The blowup was indicated by the method during the computation.

When conducting the experiment with altered $\mu = 0.2$, no blowup occurs, as can be seen in Figure \ref{fig:exp431prolif}. The aggregation stops and reverses since the logistic term attracts the cell concentration to a lower density. The total mass of the cells decreases after the aggregation stops and increases again after around $t=1.5$. No blowup could be observed even for later times, instead the numerical solution seems to converge to a stationary state. The CFL condition given by \eqref{eq:splschemecfl} has caused an increase of the time increment over the computation time.

\begin{figure}
	\begin{tabular}{c c}
		\includegraphics[width=0.5\linewidth]{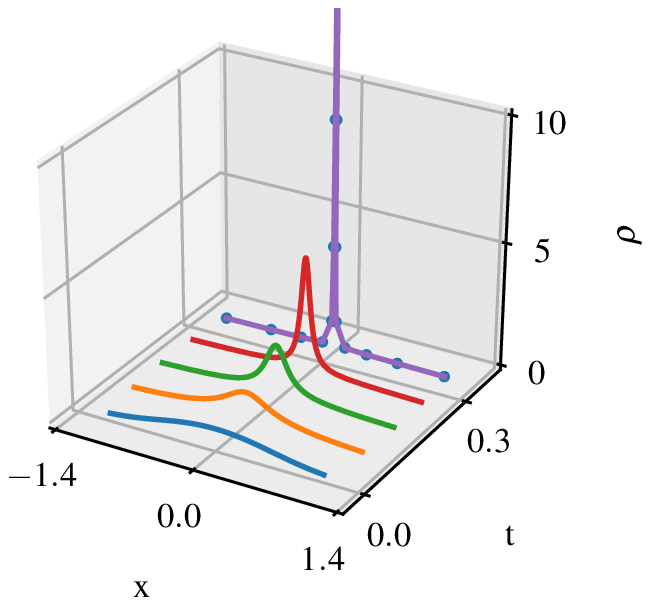} & \includegraphics[width=0.5\linewidth]{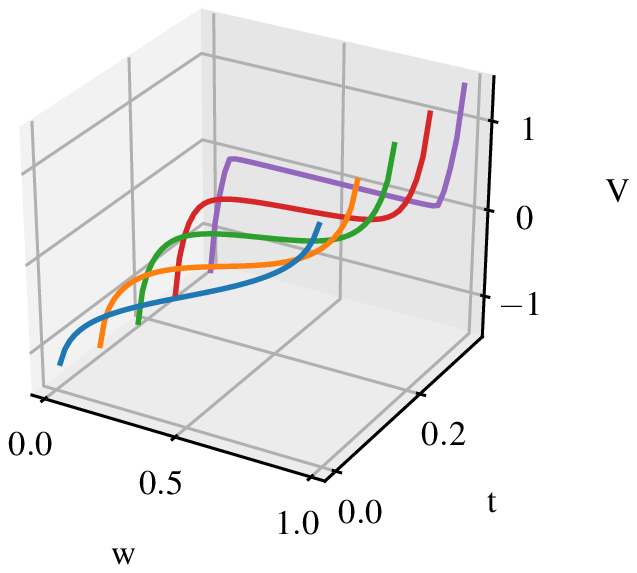}\\
	\end{tabular}
	\caption{Numerical results (cell concentration and inverse cumulative function) for the parabolic elliptic Keller-Segel model, experiment 4.3.1 in \cite{Carillo.2008}. The cell concentration blows up. The numerical cell concentration has attained a maximum of approximately $176$. }\label{fig:exp431}
\end{figure}

\begin{figure}
	\centering
\includegraphics[trim=0 0.5cm 0 0.5cm, clip]{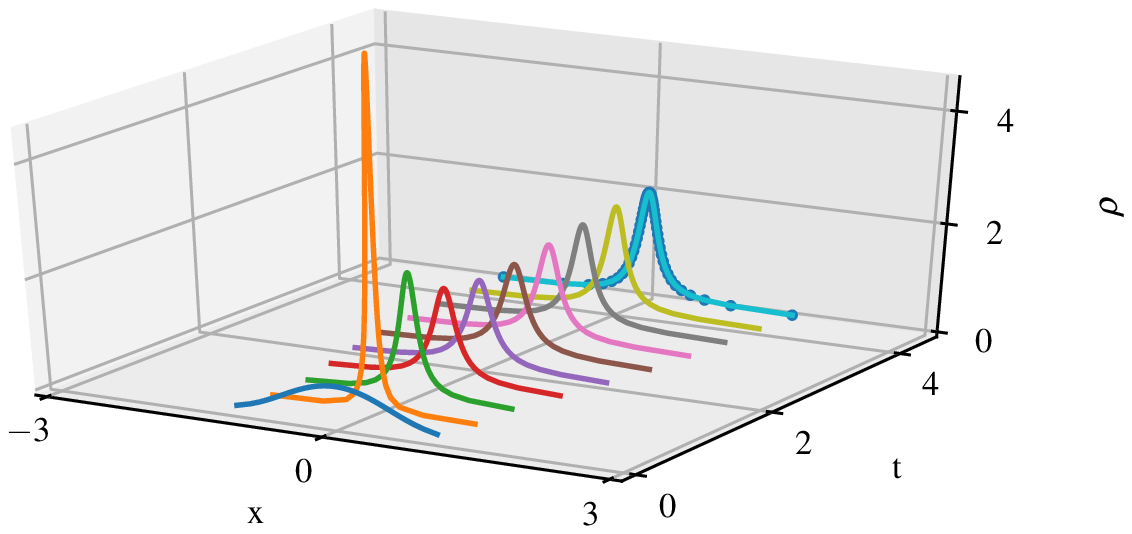}
\includegraphics[trim=0 0 0 0.5cm, clip]{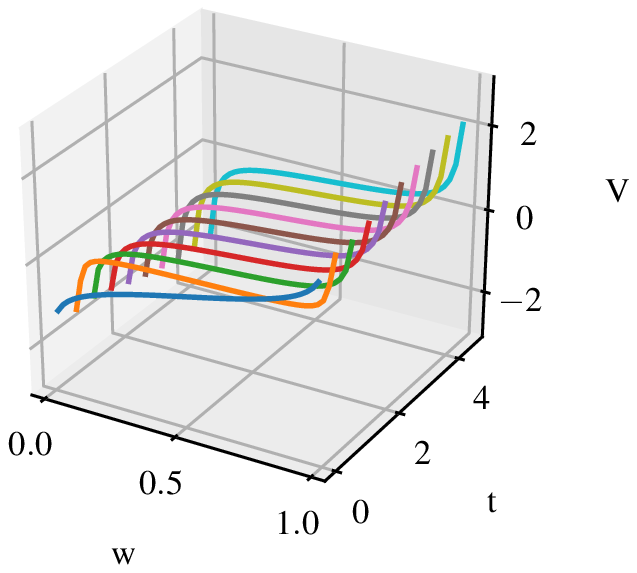}
\includegraphics[scale = 0.85]{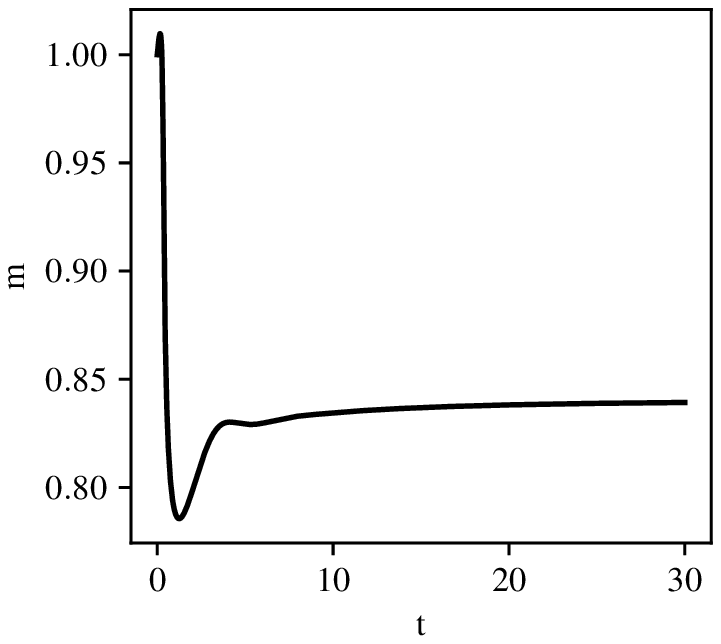}
	\caption{Numerical results (cell concentration, inverse cumulative function, and mass) for the parabolic elliptic Keller-Segel model with added logistic growth \eqref{eq:PKS_lgg}. The additional reaction term has prevented blowup.} \label{fig:exp431prolif}
\end{figure}

\subsection{Nonlinear diffusion and chemotaxis models}
Our method can also resolve models that include generalized diffusion and migration terms as we will demonstrate in this section. To this end we consider at first the model
\begin{equation}\label{eq:PKS_varm}\left\{
\begin{aligned}
\ddt \rho &= \ddx \( \rho^{\gamma-1} \ddx \rho - \chi  \rho \ddx c\)& \text{in }(0,\infty) \times \real, \\
c(\cdot, x) &= - \frac{1}{\pi} \int_\real \log(|x-y|)\rho(\cdot,y)\,dy &\text{in }(0,\infty) \times \real,\\
\rho(0,\cdot) &= \rho_0 \geq0
\end{aligned}\right.
\end{equation}
with exponent $\gamma>1$. In the case $\chi=0$ the first equation in \eqref{eq:PKS_varm} is known as the \emph{porous media equation} modeling the gas flow through a porous interface. We refer to \cite{Vazquez.2007} for an introduction to the subject.

Similar as in \eqref{eq:V_scheme_IKS}, the scheme to resolve \eqref{eq:PKS_varm} corresponds to \eqref{eq:V_scheme_IKSLnD}-\eqref{eq:V_scheme_IIKSLnD} where the chemo-attractant gradient is computed as
\begin{equation}\label{chemoconv}
	- \ddx \hat c_h(V_j(t)) = \frac{\Delta_w(t)}{\pi} \lim\limits_{\varepsilon \rightarrow 0} \sum_{i: |V_j(t) -  V_i(t)|\geq \varepsilon} \frac{1}{ V_j(t)- V_i(t)}.	
\end{equation}

\begin{figure}
	\centering
	\includegraphics[scale=0.95, trim=0 0 0 1cm, clip]{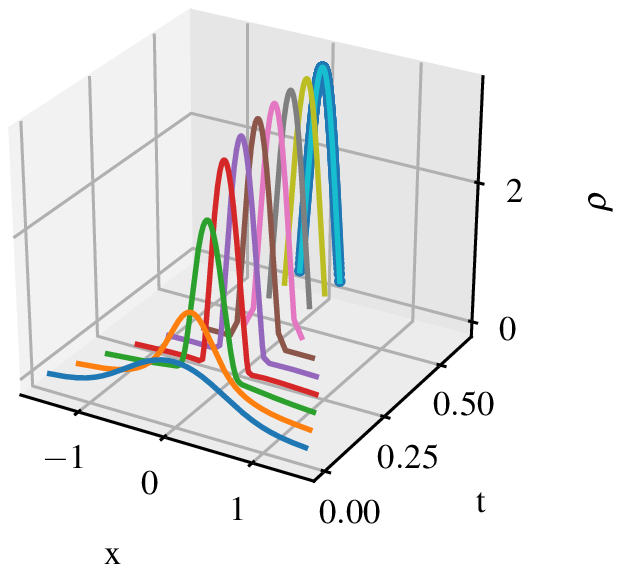}  \includegraphics[scale=0.95, trim=0 0 0 1cm, clip]{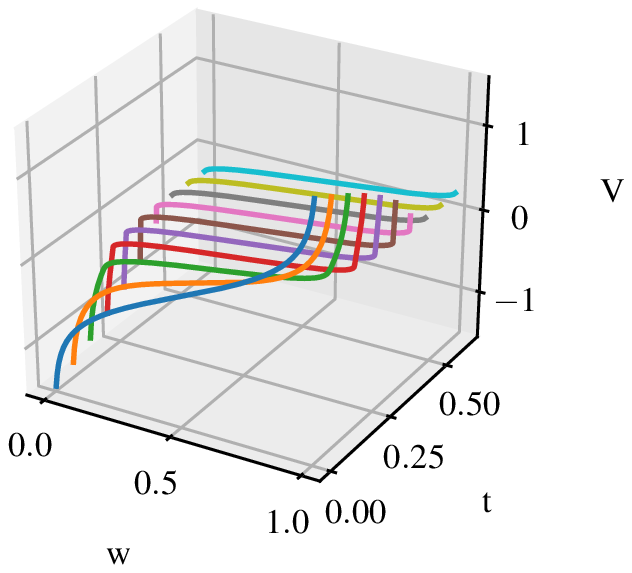}
	\includegraphics[scale=0.95, trim=0 0 0 1cm, clip]{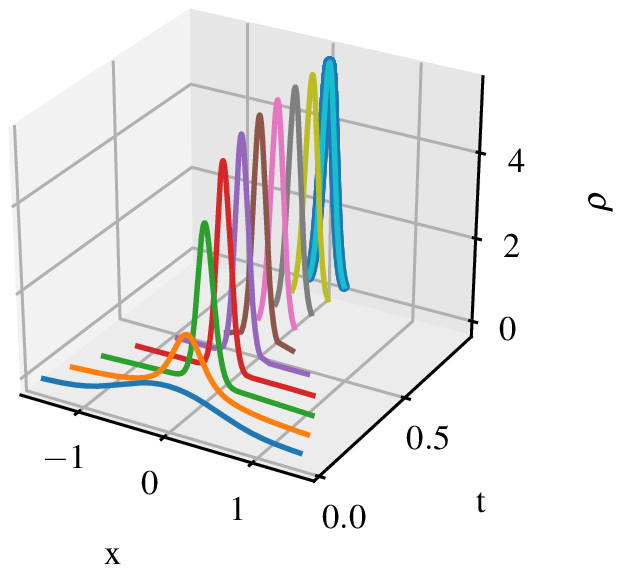}  \includegraphics[scale=0.95, trim=0 0 0 1cm, clip]{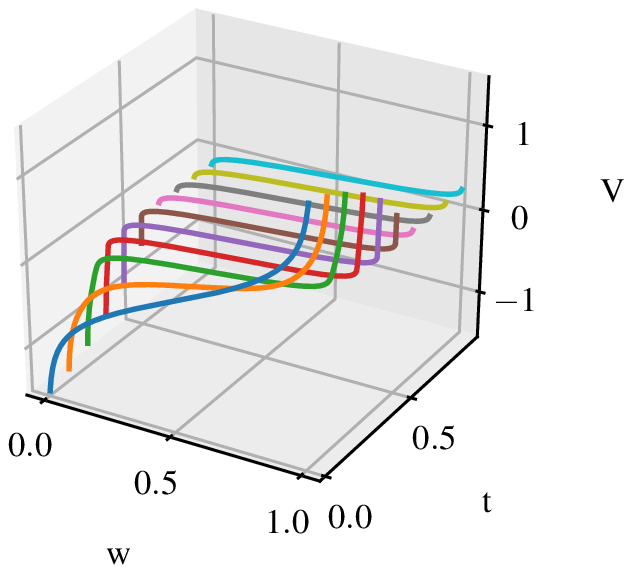}
	\caption{Numerical results (cell concentration and inverse cumulative function) for the nonlinear diffusion model \eqref{eq:PKS_varm}, initial condition \eqref{eq:V_singlePeak_ini}, $\chi=2.5\pi$ and $\gamma=2$ (top row), $\gamma=1.5$ (bottom row). For both chosen values of $m$ the numerical solution converges to a steady state. We have used $M=500$ points.}\label{fig:nonlinear_diff}
\end{figure}

We have tested our scheme using again the initial condition \eqref{eq:V_singlePeak_ini} and the chemo-sensitivity $\chi=2.5\pi$. Figure \ref{fig:nonlinear_diff} exhibits the results from the numerical simulation for the exponents $\gamma =2$ and $\gamma=1.5$. In both cases the nonlinear diffusion prevents the blowup that would occur for $\gamma =1$ and the numerical solution converges to a stationary state.

Another model that we consider here has been proposed in \cite{HPVolume}. In this work the authors endowed the Keller-Segel model with a volume filling mechanism. For this purpose they reconsidered the derivation of the model from a random walk and added a function $q(\rho)$ describing the probability that a cell finds sufficient space to jump to a particular position. We adopt here the probability function $q(\rho) = 1- \rho^\gamma$ that models the volume filling together with enhanced diffusion for $\gamma >1$ and reduced diffusion for $\gamma<1$ \cite{HPVolume}. Independent of the choice of $\gamma>1$, the model does not allow for cell migration to a position where the maximal density $\rho=1$ has already been reached. The corresponding model for the cell density includes nonlinear diffusion and advection terms and reads
\begin{equation}\label{eq:vol_filling}
	\ddt \rho = \ddx \( D_\rho (1 + (\gamma -1)\rho^\gamma) \ddx \rho - \chi (1-\rho^\gamma) \rho \ddx c\)\qquad \text{in }(0,\infty) \times \real.
\end{equation}

For the numerical experiments with the volume filling model \eqref{eq:vol_filling} we have adapted the update steps \eqref{eq:V_scheme_IKSLnD} and \eqref{eq:V_scheme_IIKSLnD} in the diffusion-advection operator by
\begin{subequations}\label{eq:vol_filling_scheme}
	\begin{align*}
	- 2\frac{\tilde V_j (t) - V_j(t)}{\Delta t} &= \frac{1}{\tilde V_{j+1}(t) - \tilde V_{j}(t)} - \frac{1}{(\tilde V_j(t) - \tilde V_{j-1}(t))}
	+\frac{(\gamma -1)\tilde D(t)}{(\tilde V_{j+1}(t) - \tilde V_{j}(t))^{\gamma}} - \frac{(\gamma -1)\tilde D(t)}{(\tilde V_j(t) - \tilde V_{j-1}(t))^{\gamma}} \nonumber \\[1ex]
	& \quad - \chi  \left[1 - \(\frac{2 \, \Delta w}{V_{j+1}(t) - V_{j-1}(t)}\)^\gamma\right] \ddx \hat c_h( V_j(t)),	
	\end{align*}
	\begin{align*}
	- \frac{T_{\Delta t} V_j (t) - V_j(t)}{\Delta t} &= \frac{1}{\tilde V_{j+1}(t) - \tilde V_{j}(t)} - \frac{1}{(\tilde V_j(t) - \tilde V_{j-1}(t))}
	+\frac{(\gamma -1)\tilde D(t)}{(\tilde V_{j+1}(t) - \tilde V_{j}(t))^{\gamma}} - \frac{(\gamma -1)\tilde D(t)}{(\tilde V_j(t) - \tilde V_{j-1}(t))^{\gamma}} \nonumber \\[1ex]
	& \quad - \chi  \left[1 - \(\frac{2 \, \Delta w}{V_{j+1}(t) - V_{j-1}(t)}\)^\gamma\right] \ddx \hat c_h( \tilde V_j(t)),
	\end{align*}
\end{subequations}
where we set $\frac{2 \, \Delta w}{V_{j+1}(t) - V_{j-1}(t)}=0$ for $j=1$ and $j=M$ to account for the boundary conditions.
\begin{figure}
	\centering
	\includegraphics[scale=0.95, trim=0 0 0 1cm, clip]{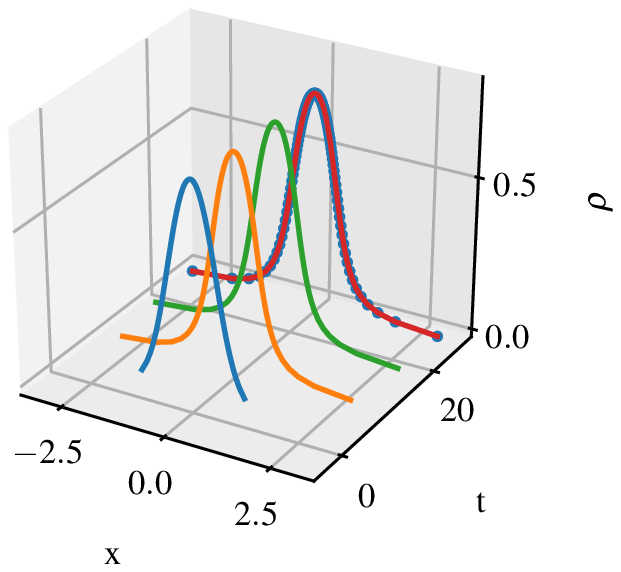}  \includegraphics[scale=0.95, trim=0 0 0 1cm, clip]{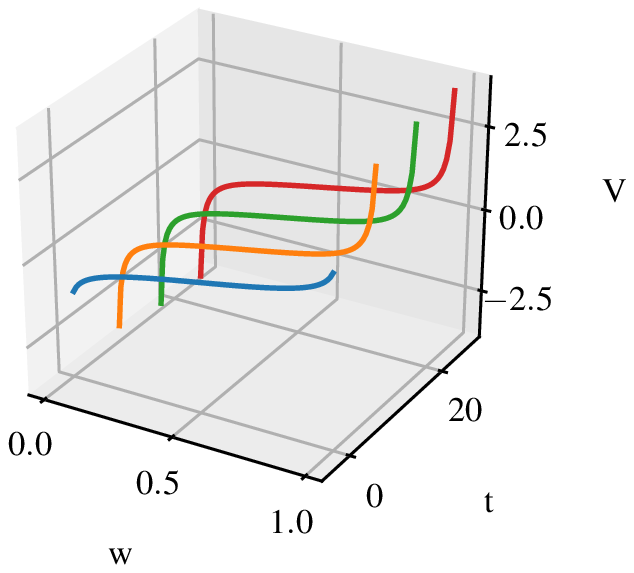}
	\includegraphics[scale=0.95, trim=0 0 0 1cm, clip]{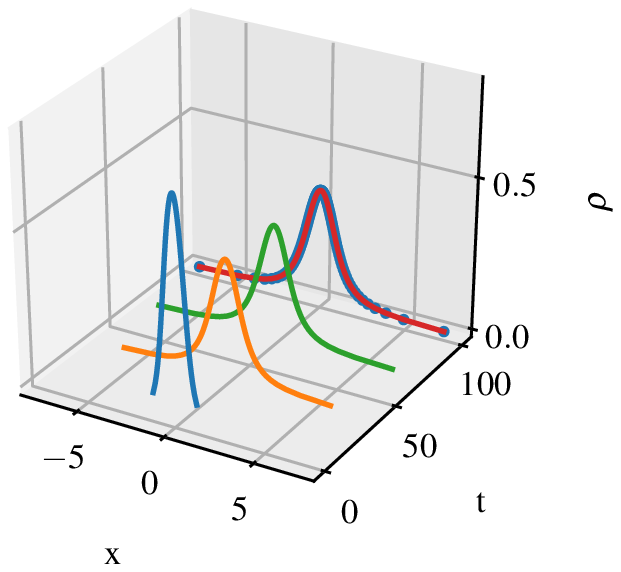}  \includegraphics[scale=0.95, trim=0 0 0 1cm, clip]{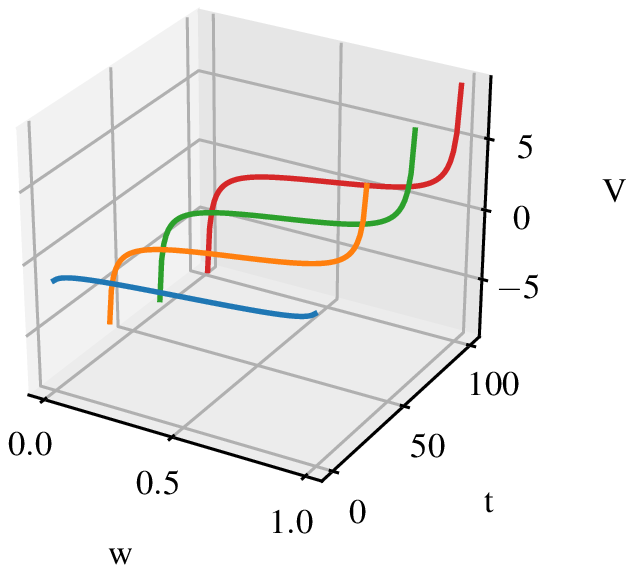}
	\caption{Numerical results (cell concentration, inverse cumulative function) for the model \eqref{eq:vol_filling} with chemo-attractant given by a convolution as in \eqref{eq:PKS_lgg}. Results are shown for $\gamma=2$ (top row) and $\gamma=0.5$ (bottom row). We have used $M=50$ points in the numerical computation.}\label{fig:vol_filling}
\end{figure}

In Figure \ref{fig:vol_filling} we present simulation results with the parabolic-elliptic model \eqref{eq:PKS_varm} where we have replaced the original evolution equation of the cell density by the volume filling approach \eqref{eq:vol_filling}. Again we have used the initial condition \eqref{eq:V_singlePeak_ini} and the chemo-sensitivity parameter $\chi=2.5\pi$. We exhibit the numerical results for $\gamma=2$ and $\gamma=0.5$. The computed cell densities do not exceed a density of one in both cases and no blowup occurs. Instead the cells evolve quickly to a bounded spatial profile from which they slowly diffuse afterwards. The parameter choice $\gamma = 2$ leads to a larger maximal cell density throughout the computation when compared to the case $\gamma=0.5$.

\subsection{The parabolic-parabolic Keller-Segel model}
In this section we apply our scheme to the well known parabolic-parabolic Keller-Segel system which reads
\begin{equation}\label{eq:PKS_gf}\left\{
\begin{aligned}
\ddt \rho &= \ddx \( D_\rho \ddx \rho - \chi  \rho \ddx c\),&\text{in }(0,\infty) \times (a,b), \\
\ddt c &= D_c \ddx^2 c +  \rho -  c, &\text{in }(0,\infty) \times (a,b),\\
	\ddx \rho(\cdot, r) &= \ddx c(\cdot, r)=0, &r\in\{a,b\},\\
\rho(0,\cdot) &= \rho_0 \geq0,\quad c(0,\cdot)= c_0 \geq 0.&
\end{aligned}\right.
\end{equation}
As opposed to \eqref{eq:PKS_lgg} this system features an additional parabolic equation to be treated by the splitting method.
In order to exhibit the phenomena that the scheme can resolve, we consider two test cases with distinct initial chemo-attractant densities that control the cell movement. In both tests we adopt the initial datum \eqref{eq:V_singlePeak_ini} for the inverse distribution $V$.

\subsubsection*{Peak movement}
For our first numerical experiment with the system \eqref{eq:PKS_gf} we use the parameters $D_\rho =0.1, ~D_c = 0.01,~\chi = 2.5,~\alpha = 0.5,~\beta = 1$ and the domain $\Omega=(a,b)$ with boundaries chosen $a=V_0(0)\approx-1.58,~b=V_0(1)\approx1.58$. As initial chemo-attractant concentration we take the logistic function
\[ c_0(x) = \frac{1}{1+ \e^{-5\, x}}, \quad x \in \Omega.\]  For the simulation we employ meshes with $M=45$ and $N=230$ points and the CFL condition \eqref{eq:gf_cfl}.

Figure \ref{fig:movement} presents the cell dynamics, showing their movement to the right side of the domain. As the cells produce the chemical with density $c$, a negative gradient is created that leads to an aggregation of the cells which counteracts the movement. We point out that both the migration and the growth are well resolved by the splitting scheme.

\begin{figure}
	\centering
	\includegraphics[trim=0 0.5cm 0 1cm, clip]{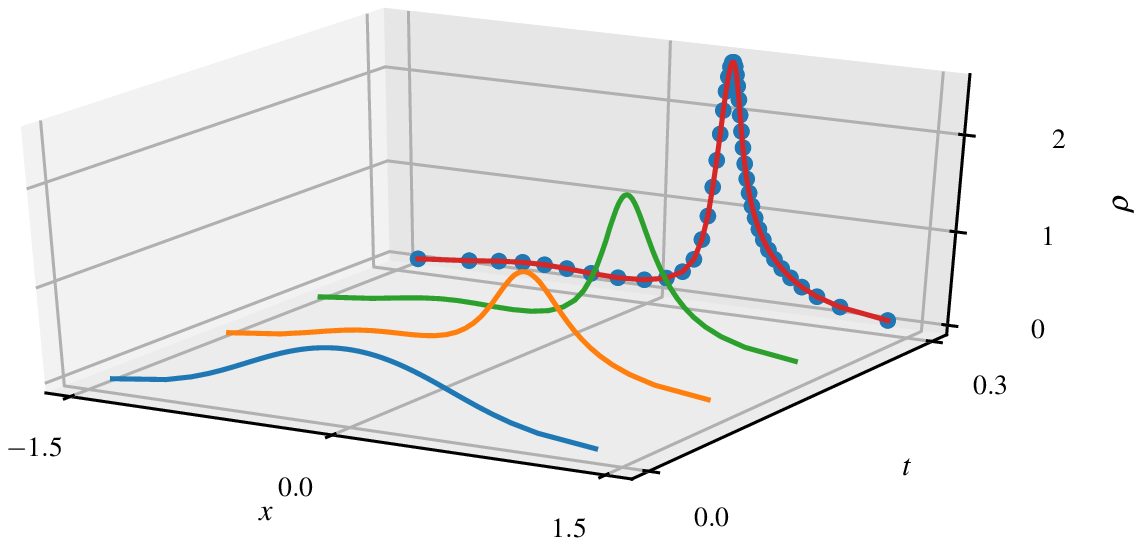}
	\includegraphics[scale=0.98, trim=0 0 0 1cm, clip]{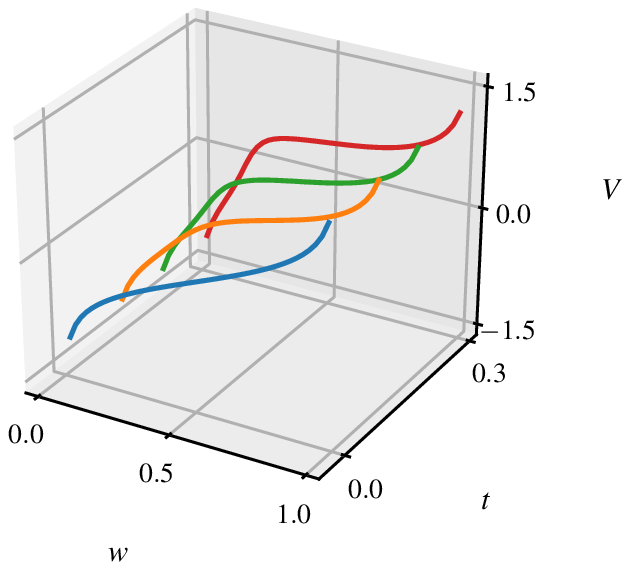}
	\includegraphics[scale=0.98, trim=0 0 0 1cm, clip]{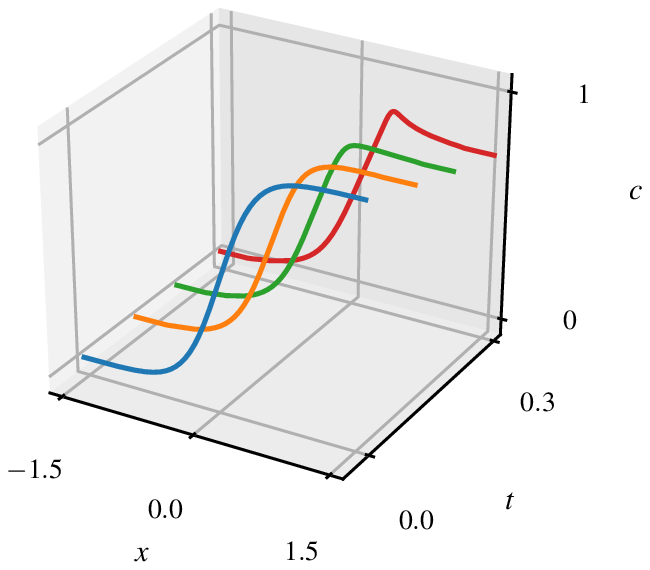}
	\caption{Numerical results (cell concentration, inverse cumulative function, and chemo-attractant density in space and time) for the parabolic-parabolic KS model \eqref{eq:PKS_gf} in the \enquote{peak movement} experiment. The movement and aggregation is accurately resolved using $M=45$ grid points.}\label{fig:movement}
\end{figure}

\subsubsection*{Peak splitting}
In the next test we use the parameters $D_\rho =D_c = 0.1,~\chi = 5,~\alpha = \beta = 1$ and the computational domain $(a,b)$ with boundaries chosen as in the \enquote{peak movement} experiment. The initial chemo-attractant density though is replaced by the function
\[c_0(x) = 1 - \e^{-20\, x^2}, \quad x \in (a,b).\]
Figure \ref{fig:splitting} shows the computational results for the discretization parameters $M=90$ on the mass space mesh and $N=450$ on the Finite Element mesh. The cells move out of the center of the domain on which the most part of the attracting chemical is already consumed. The symmetrical movement to both sides leads to a splitting of the initial concentration into two peaks. The discretization grid for the cells on the density level concentrates its grid points on the locations of both peaks and adapts to the solution over time.

\begin{figure}	
	\centering
	\includegraphics[trim=0 0.5cm 0 1cm, clip]{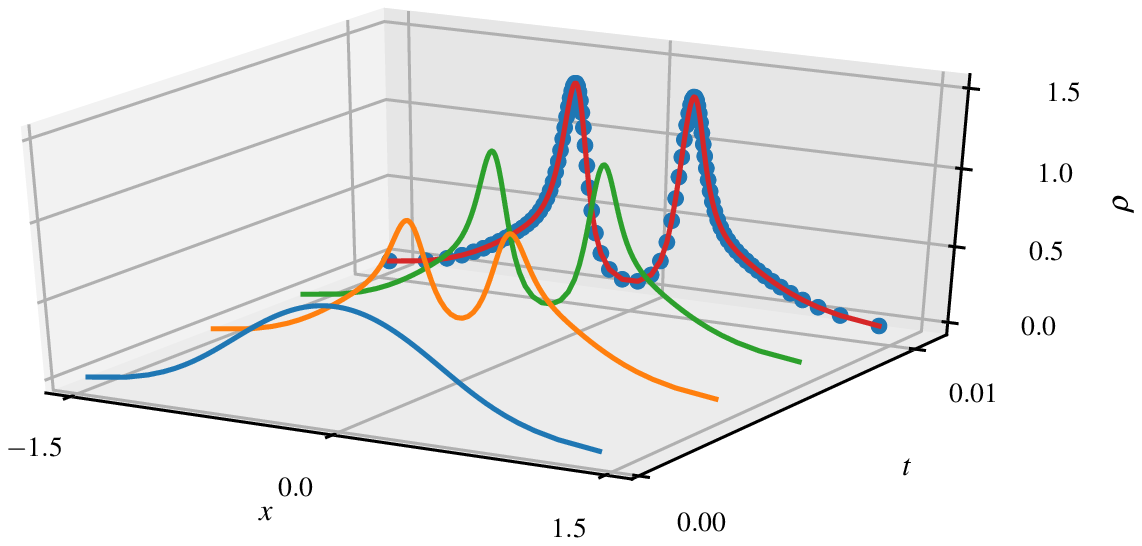}
	\includegraphics[scale=0.98, trim=0 0 0 1cm, clip]{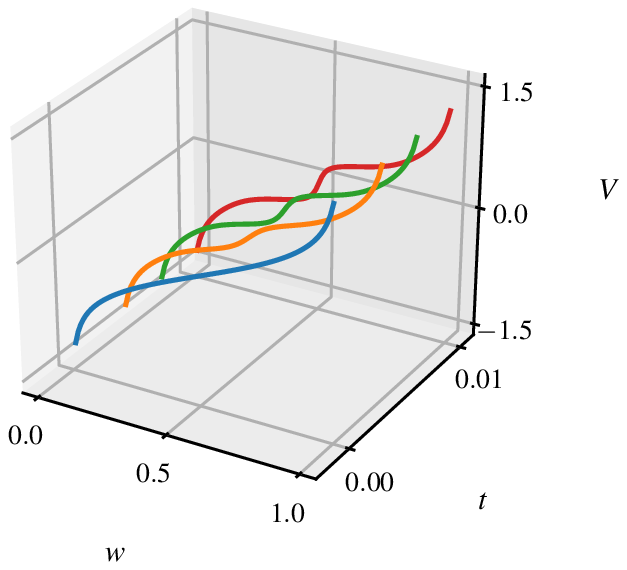}
	\includegraphics[scale=0.98, trim=0 0 0 1cm, clip]{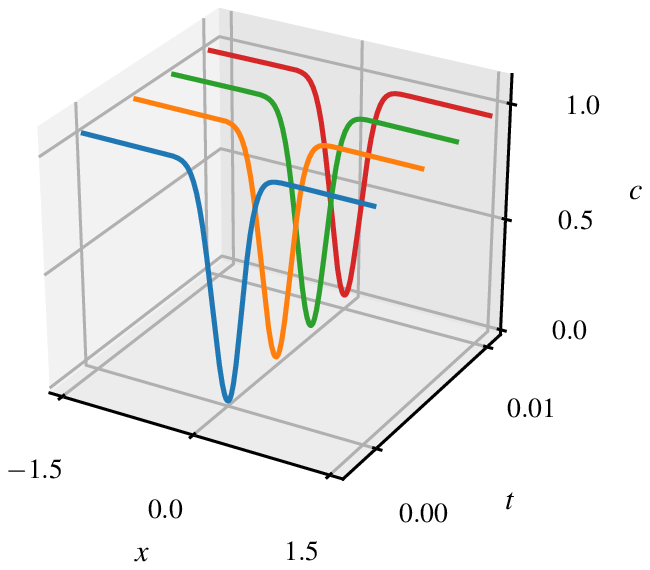}
	\caption{Numerical results (cell concentration, inverse cumulative function, and  chemo-attractant density in space and time) for the parabolic-parabolic KS model \eqref{eq:PKS_gf} in the \enquote{peak splitting} experiment. At the final time we show the approximated cell averages of the density $\rho$ at their respective position on the grid (top). The grid for the cell density adapts to the two splitting peaks.}\label{fig:splitting}
\end{figure}

In the setting of the present experiment we study the convergence of the introduced splitting scheme experimentally. For a fixed instance in time and given $M$, let $V^h_i,~i=1,\dots,M-1$ denote a numerical solution corresponding to the mesh discretization parameter $M$. Then we define the approximate $L^1$ finite difference error by
\begin{equation}\label{eq:fd_error}
E^V_M = \frac{1}{M} \sum_{j= 1}^{M-1}  \left| V^{h}_j - V^{h/2}_{2j} \right|,
\end{equation}
where we have used a numerical solution on a finer mesh with $2M$ points, $V^{h/2}_j,~j=1,\dots,2M-1$, as the reference solution. The \emph{experimental order of convergence} (EOC) for the discretization error in $V_h$ can now be defined by
\begin{equation}\label{eq:eoc_v}
EOC^V(M) = log_2(E^V_M) - \log_2(E^V_{M/2})
\end{equation}
for any even integer $M$. Similarly, we define the EOC for the cell densities on their non-uniform mesh. To this end let $\rho^h_i,~i=1,\dots,M$ denote the finite volume representation corresponding to $V^h_i,~i=1,\dots,M-1$ and let $E^\rho(M)$ denote the discrete $L^1$ error using as reference $\rho^\text{ref}_i,~i=1,\dots,2M$ the finite volume representation of $V^{h/2}_j,~j=1,\dots,2M-1$. This $L^1$ error is computed by projecting the reference solution to the coarser non-uniform grid corresponding to the cell densities $\rho^h_i,~i=1,\dots,M$. Then we define for even integers $M$ analog to \eqref{eq:eoc_v}
\[EOC^\rho(M) = log_2(E^\rho_M) - \log_2(E^\rho_{M/2}).\]

\begin{table}
	\centering	
	\begin{tabular}{r | r r |r r}
		\Tstrut \Bstrut $M$/$M\rfr$& error $E^V_M$ & $EOC^V$ & error $E^\rho_M$ & $EOC^\rho$ \\ \hline
		20 / 40 &  7.231e-04 &  & 1.360e-02 &   \Tstrut\\
		40 / 80 &  7.427e-05 & 3.283 & 1.719e-03 & 2.984 \\
		80 / 160 &  1.416e-05 & 2.391 & 2.768e-04 & 2.635 \\
		160 / 320 &  4.698e-06 & 1.592 & 7.085e-05 & 1.966 \\
		320 / 640 &  1.615e-06 & 1.541 & 2.591e-05 & 1.451 \\
		%640 / 1280 &  6.919e-07 & 1.223 & 1.554e-05 & 0.738 \\
	\end{tabular}\vspace{2ex}
	\caption{Mesh convergence in the peak splitting experiment up to $T = 0.01$ with respect to the discretization parameter $M$. We have adapted the number of points on the Finite Element mesh by $N= M$. The EOCs approach two in the inverse distributions and in the corresponding densities. Yet, for large $M$ the EOC drops which is probably due to limitations by the Finite Element mesh.} \label{tbl:EOC_peak_splitting}
\end{table}
In Table \ref{tbl:EOC_peak_splitting} we exhibit the errors and the EOCs computed at the final time $T=0.01$ and with constant time increment $\Delta t = 10^{-4}$ when doubling the mesh resolution on the mass space mesh iteratively. We have coupled the resolution of the Finite Element mesh to the number of points for the inverse distribution by using $N= M$. We can clearly see that the method converges as the mesh size is refined. The EOCs indicate a convergence order of two in both, the inverse distributions and the densities. However, we see that the EOC decreases as the grids  becomes very fine. We suppose that this is caused by the finite element mesh that is only uniformly but not locally refined: as $M$ increases the number of mesh cells on the non-uniform finite volume mesh for the cell densities aggregates around the positions of the peaks. Throughout the computation the finite element solution $c_h$ must in turn be interpolated in many points in a small physical area which leads to a loss of accuracy as the number $\max_i|\{j: x_i\leq V_j\leq x_{i+1}\}|$ increases. Nevertheless, Table \ref{tbl:EOC_peak_splitting} demonstrates that the method has provided accurate numerical results using only a few mesh points.

\subsection{A cancer invasion system}\label{sec:cancer1}
In this test case we address a model of cancer invasion of the extracellular matrix (ECM), the first step in cancer metastasis. The macroscopic modeling of this process commonly uses an Keller-Segel approach that models the densities of the cancer cells, the concentration of the extracellular fibers on which cancer cells adhere and move, and the density of an enzyme of the matrix metallopreteinases (MMPs) family that is produced by the cancer cells and is responsible for the degradation of the ECM.

There is a wide variety of cancer invasion models in the literature, see e.g. \cite{Chaplain.2005, EMT_paper,  Preziosi.2003, Stinner.2015, Johnston.2010}. In order to test our scheme we employ a simple test case based on the pioneering model \cite{Anderson.2000} augmented with a proliferation term in the cancer cell density equation which reads

\begin{equation}\label{eq:splitting_invasion}\left\{
\begin{aligned}
\ddt \rho &= \ddx \( D_\rho \ddx \rho - \chi  \rho \ddx v\)+ \mu \rho(1-\rho)&\text{in }(0,\infty) \times (a,b), \\
\ddt v &= - \delta v m &\text{in }(0,\infty) \times (a,b), \\
\ddt m &= D_m \ddx^2 m + \alpha \rho -  \beta m &\text{in }(0,\infty) \times (a,b)\\
	\ddx \rho(\cdot, r) &= \ddx m(\cdot, r)=0, &r\in\{a,b\},\\
\rho(0,\cdot) &= \rho_0 \geq0,\quad v(0,\cdot)=v_0, \quad m(0,\cdot)= m_0. &
\end{aligned}\right.
\end{equation}

 In this model the cancer cells with the density $\rho$ move using their motility apparatus with a preferred direction towards higher concentrations of the ECM with concentration denoted by $v$. This is the haptotaxis phenomenon. Being a network in a static equilibrium the ECM does not translocate. The MMPs however, whose density we denote by $m$, diffuse freely in the extracellular environment. Additionally, the cancer cells proliferate towards a preferred density $\rho=1$ and they produce MMPs with a constant rate. The MMPs attach to the ECM which they dissolve upon contact.

\begin{figure}
		\centering
\includegraphics[scale=0.95, trim=0 0 0 1cm, clip]{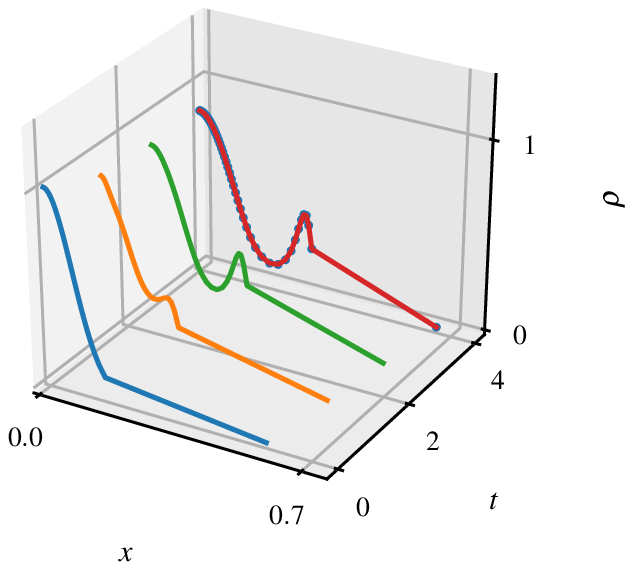}  \includegraphics[scale=0.95, trim=0 0 0 1cm, clip]{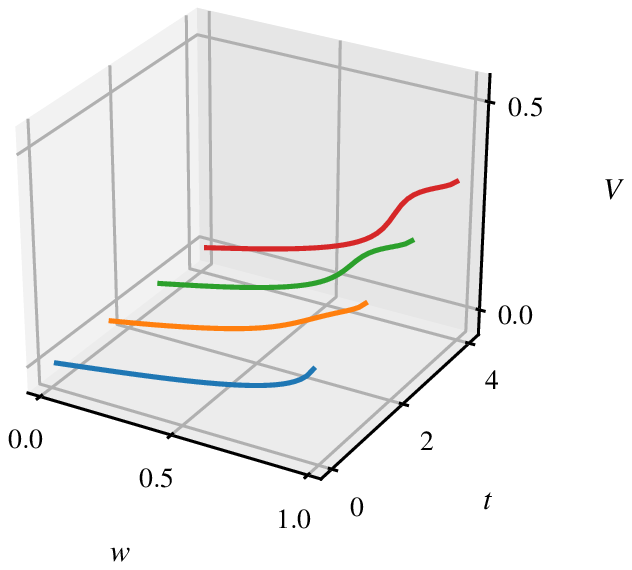}
\includegraphics[scale=0.95, trim=0 0 0 1cm, clip]{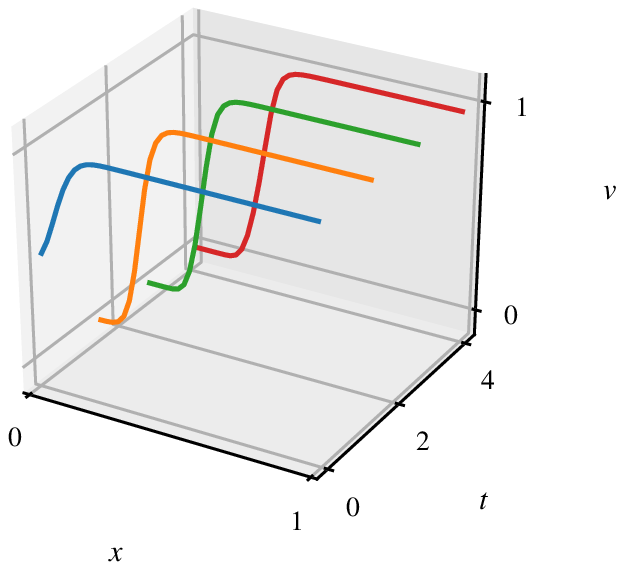}  \includegraphics[scale=0.95, trim=0 0 0 1cm, clip]{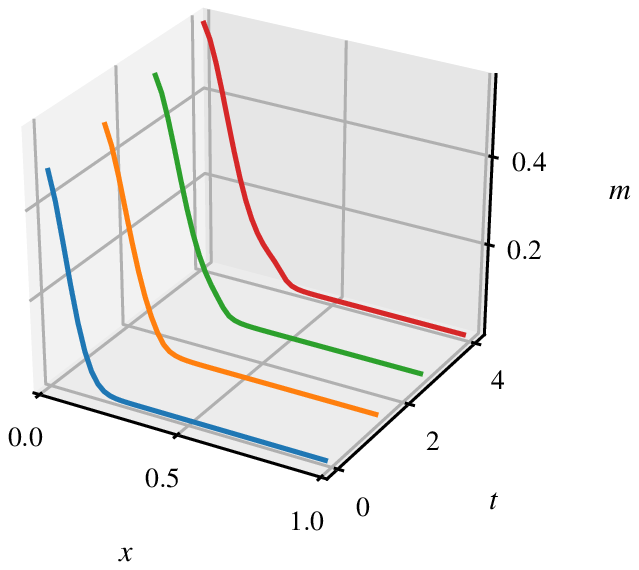}
	\caption{Numerical results (cell concentration, inverse cumulative function, tissue and MMP density) for the cancer invasion model \eqref{eq:splitting_invasion} using only $N=M=45$ mesh points.  At the final time we show the approximated cell averages of the tumor density $\rho$ at their respective position on the grid (top left). A high concentration of tumor cells emerges and invades the tissue. The grid for the tumor density omits the part of the tissue that is not yet invaded.}\label{fig:gradflow_cancer}
\end{figure}

For a numerical experiment with this model we consider the computational domain $(0,1)$ with initial conditions
\begin{equation}
		\rho_0(x) =  \e^{-x^2 / \varepsilon}, \quad
		v_0(x) = 1- 0.5\, \e^{-x^2 / \varepsilon},\quad
		m_0(x) = 0.5 \, \e^{-x^2 / \varepsilon},
\end{equation}
where we use $\varepsilon = 10^{-2}$. Moreover we employ the parameter values $D_c = 2 \cdot 10^{-4}, ~ \chi = 5 \cdot 10^{-3},~\mu = 0.2,~D_a= 10^{-3}, ~\delta = 10, ~\alpha = 0.1,$ and $\beta = 0$. We apply the splitting scheme \eqref{eq:full_strang} using meshes of $M=N=45$ points and the CFL condition \eqref{eq:splschemecfl}.

In our method we discretize both the ECM density $v$ and the MMP concentration $m$ on the same finite element basis. The corresponding approximations are updated in the reaction-diffusion operator of the splitting method. The interpolations are only needed with respect to the ECM density $v$. We resolve the migration of the cancer cells in transformed variables with the advection-diffusion operator and the cell proliferation in original variables with the reaction-diffusion operator.

The considered numerical experiment simulates the propagation of cancer cells into the ECM on the right side of the computational domain. To account for the corresponding temporal expansion of the support of the cancer cell density $c$ we have adapted the treatment of the right boundary. In more details, we have neglected the discrete cancer cell density entry adjacent to the right boundary in the proliferation update \eqref{eq:reac_update}, i.e. $S_{\Delta t}\rho_M(t) = \rho_M(t) $. Though we have not excluded the corresponding boundary entry in the cumulative function, $V_{M-1}$, from the diffusion and haptotaxis updates of the scheme.

We present the according numerical results in Figure \ref{fig:gradflow_cancer}. Apart from the propagation of the cells into the tissue, we observe a build up of cancer cells at the leading front of the tumor. Degradation of the tissue and MMP production are also visible. Throughout the computation the not invaded part of the tissue is resolved by a single grid cell in the cancer cell density.

\begin{table}
	\centering	
	\begin{tabular}{r | r r }
		\Tstrut \Bstrut	$\Delta t$/$\Delta t\rfr$& error $E^V_{\Delta t}$ & $EOC^t$  \\ \hline
		\Tstrut	0.1/ 0.05 &  2.244e-04 & \\
		0.05/ 0.025 &  4.728e-05 & 2.247   \\
		0.025 / 0.0125 &  1.107e-05 & 2.094   \\
		0.0125 / 0.00625 &  2.766e-06 & 2.001  \\
	\end{tabular}\hspace{2em}
	\begin{tabular}{r | r r}
		\Tstrut \Bstrut	$M$/$M\rfr$& error $E^V_M$ & $EOC^V$  \\ \hline
		\Tstrut 10 / 20 & 7.867e-03 & \\
		20 / 40 &  1.919e-03 & 2.035 \\
		40 / 80 &  4.475e-04 & 2.100\\
		80 / 160 &  1.514e-04 & 1.563 \\
	\end{tabular}\vspace{2ex}
	\caption{Experimental convergence in time (left) and in space (right) in the numerical experiment with system \eqref{eq:splitting_invasion} at $T = 1$. In all computations we have set $N= M$. The EOCs suggest a convergence of second order in time and space.} \label{tbl:EOC_invasion_gf}
\end{table}
In this experiment we have also studied the convergence of the scheme experimentally. Along with the errors in space, we have also computed the errors in time by the formula
\begin{equation}\label{eq:fd_error_time}
E^V_{\Delta t} = \frac{1}{M} \sum_{j= 1}^{M-1}  \left| V^{h,\,\Delta t}_j - V^{h,\,2\Delta t}_{j} \right|,
\end{equation}
where $V^{h,\,\Delta t}_i,~i=1,\dots,M-1$ denotes a numerical solution computed on $M$ mesh points with constant time increment $\Delta t$. For the computation of the temporal errors we have considered a fine spatial resolution with $M=N=600$ mesh cells. The corresponding EOC is given by $EOC^t=log_2(E^V_{\Delta t}) - \log_2(E^V_{2\Delta t})$. The spatial errors and EOCs are computed according to \eqref{eq:fd_error} and \eqref{eq:eoc_v} with constant time increment $\Delta t = 2 \times 10^{-4}$ and coupled $N=M$. Both, temporal and spatial errors have been computed at the final time $T=1$.

In Table \ref{tbl:EOC_invasion_gf} we present the computed errors and EOCs in the invasion experiment. We see that the method converges as either the mesh size or the time increment is refined. The EOCs in time and space range around two which confirms our expected second order. As in the \enquote{peak splitting} experiment, the EOC decreases slowly as the mesh is refined to very high resolutions. We point out that previous numerical tests which did not employ our proposed boundary treatment have yield only a spatial EOC of one.

\subsection{The uPA model}\label{sec:uPA}
In the last series of experiments we apply our scheme to a detailed tumor invasion system derived in \cite{Chaplain.2005}. This model focuses on the enzymatic \emph{urokinase plasminogen activator} (uPA) system which is known to play an essential role in the context of cancer progression and metastasis.
The uPA is an extracellular serine protease which is responsible for the activation of the protease plasmin.
This activation occurs mainly if uPA is bound to its uPA receptors (uPAR) on the cancer cell membrane. The receptor bound uPA enhances the affinity of uPAR to the ECM constituent vitronectin \cite{Wei.1994} and integrins. Thus, the uPA/uPAR-complex regulates indirectly also the vitronectin-integrin interactions. Both proteases plasmin and uPA catalyze the degradation of vitronectin and other ECM components. Another actor in the system is the plasminogen activator inhibitor type 1 (PAI-1) which is produced by the tumor cells and limits the activation of plasmin to prevent tissue damage and to maintain homeostasis.

The considered model complements the system \eqref{eq:splitting_invasion} by chemotactic movement of the cells due to uPA and PAI-1, remodeling of the ECM modeled by a logistic term and the dynamics of the uPA system modeled in terms of mass-action kinetics. We refer to \cite{Chaplain.2005} for more details. The full model reads
	\begin{equation}\label{eq:uPA}
\left\{
\begin{aligned}
\ddt \rho &= \ddx \( D_\rho \ddx \rho\right. \left. - \chi_u \rho\ddx u - \chi_p \rho \ddx p - \chi_v \rho\ddx v\) +\mu_1\rho(1-\rho) &\text{in }(0,\infty) \times (a,b),\\
\ddt v & = - \delta vm + \phi_{21}up - \phi_{22}vp + \mu_2v(1-v) &\text{in }(0,\infty) \times (a,b),\\
\ddt u &= D_u \ddx^2 u -\phi_{31}up - \phi_{33}\rho u + \alpha_3 \rho &\text{in }(0,\infty) \times (a,b),\\
\ddt p &= D_p \ddx^2 p - \phi_{41}up - \phi_{42}vp + \alpha_{4}m &\text{in }(0,\infty) \times (a,b),\\
\ddt m &= D_m \ddx^2 m  + \phi_{52} vp + \phi_{53}\rho u - \alpha_{5}m &\text{in }(0,\infty) \times (a,b),\\
	\ddx \rho(\cdot, r) &= \ddx u(\cdot, r)= \ddx p(\cdot, r)=\ddx m(\cdot, r)=0, &r\in\{a,b\},\\
\rho(0,\cdot) &= \rho_0,\quad v(0,\cdot)=v_0 ,  u(0,\cdot)= u_0\quad p(0,\cdot)= p_0, \quad m(0,\cdot)= m_0, &
\end{aligned}
\right.
\end{equation}
where the cancer cell concentration is represented by $\rho$, the ECM by the density of its constituent vitronectin $v$, and uPA, PAI-1, and plasmin densities are denoted by $u$, $p$, and $m$. We assume non-negative initial data.

We consider a numerical experiment that we have studied in \cite{Urokinase_paper} by a Finite Volume method. It employs the parameter values from \cite{Andasari.2011} given by
\begin{equation*}
\begin{array}{lll}
D_c = 3.5 \times 10^{-4},	& 	\chi_u = 3.05\times 10^{-2},		& 	\mu_1 = 0.25,	\\
D_u = 2.5\times  10^{-3},		& 	\chi_p=3.75\times  10^{-2},		&	\mu_2=0.15,		\\
D_p=3.5\times  10^{-3}, 		&	\chi_v = 2.85\times  10^{-2},		&	\delta=8.15,	\\
D_m=4.91\times  10^{-3},		&			\phi_{21}=0.75,	&	\phi_{22}= 0.55,\\	
\phi_{31}=0.75, 			&	\phi_{33}=0.3,					&	\phi_{41}=0.75, \\
\phi_{42}=0.55, &				\phi_{52}=0.11,					&	\phi_{53}=0.75, \\
\alpha_3 = 0.215, & 	\alpha_4 = 0.5, & \alpha_5=0.5,
\end{array}
\end{equation*}
and the computational domain $I=(0,10)$ with the initial date
\begin{align*}
	c_0(x) &= \e^{-x^2/\eps},
	&& v_0(x)= 1- \frac 1 2 \e^{-x^2/\eps},
	&&u_0(x) = \frac 1 2\e^{-x^2/\eps},\\
	p_0(x) &= \frac {1}{20} \e^{-x^2/\eps}, &&m_0(x) = 0, &&\varepsilon= 5 \times 10^{-3}.
\end{align*}

As done to treat the model \eqref{eq:splitting_invasion} we use a single finite element basis to discretize the concentrations of the ECM, the uPA, the PAI-1, and the plasmin. The cubic spline in the advection-diffusion operator interpolates the linear combination $\chi_v v+ \chi_u u + \chi_p p$. Similar as in the models \eqref{eq:PKS_gf} and \eqref{eq:splitting_invasion} the scheme approximates the cell proliferation in Eulerian coordinates but diffusion and advection of the cancer cells in transformed variables. We have used the same boundary treatment as in Section \ref{sec:cancer1}.

\begin{figure}[!t]
	\centering
	\includegraphics[trim=0 0.5cm 0 0.5cm, clip]{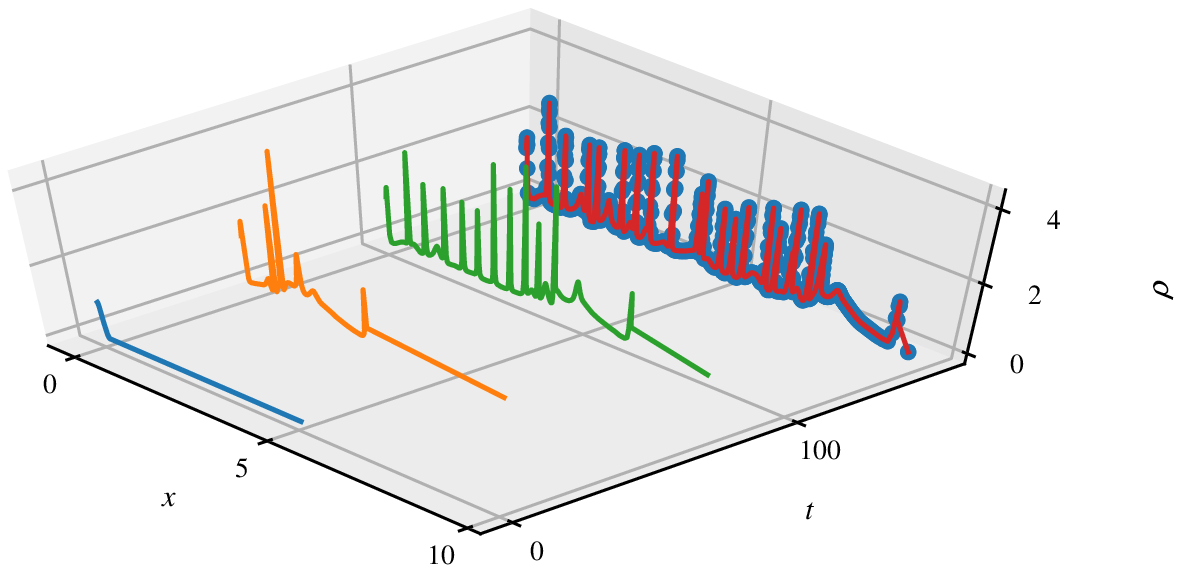}
	\includegraphics[scale=0.98, trim=0 0 0 1cm, clip]{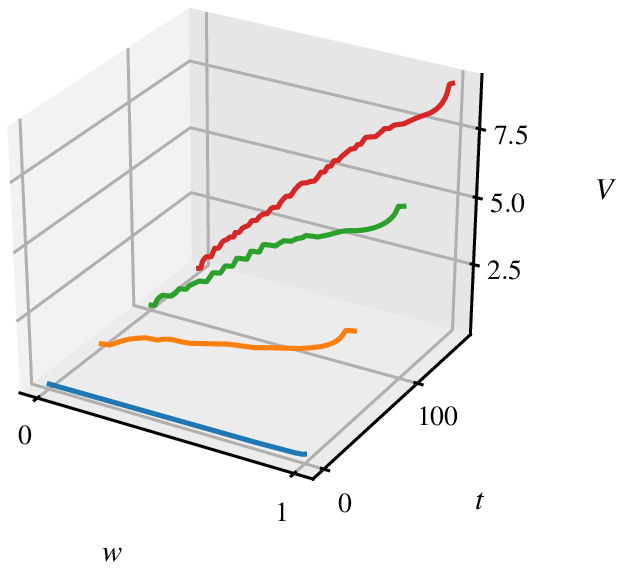}
	\includegraphics[scale=0.98, trim=0 0 0 1cm, clip]{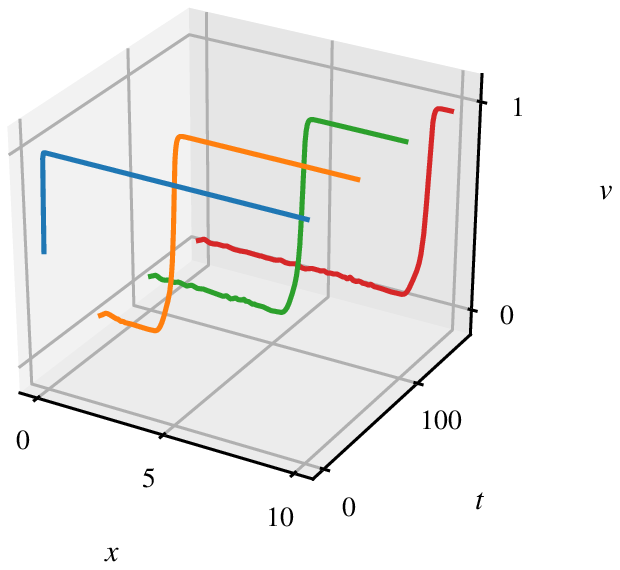}
	\includegraphics[scale=0.98, trim=0 0.2cm 0 0.5cm, clip]{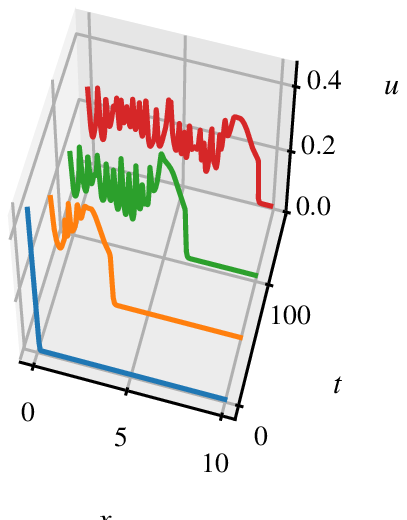}
	\includegraphics[scale=0.98, trim=0 0.2cm 0 0.5cm, clip]{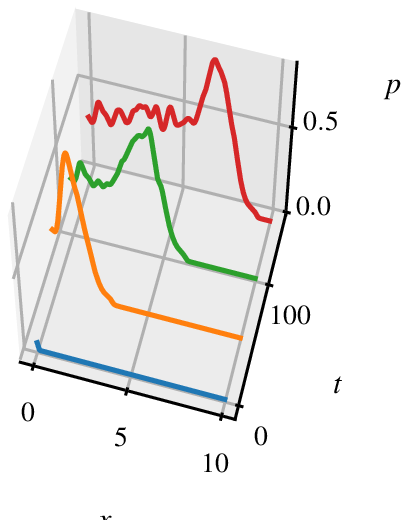}
	\includegraphics[scale=0.98, trim=0 0.2cm 0 0.5cm, clip]{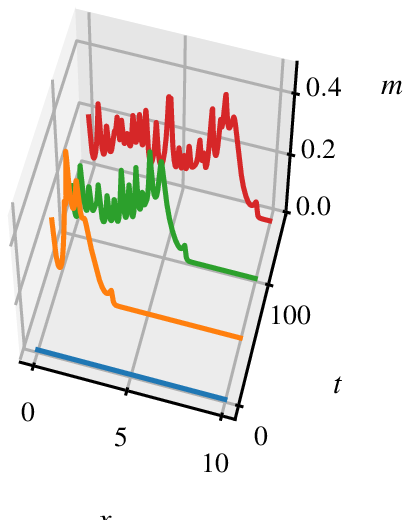}
	\caption{Numerical results (cancer cell concentration, inverse cumulative function, ECM, uPA, PAI-1 and plasmin density in space and time) in an numerical experiment with the model \eqref{eq:uPA} computed by the new scheme. The dynamics, particularly the steep peaks in the cancer cell density, are well resolved by the scheme. We have used $M=N=400$ grid points on both meshes in the numerical simulation.}\label{fig:gradflow_chaplol}
\end{figure}

In Figure \ref{fig:gradflow_chaplol} we present the simulation results obtained by our scheme with mesh parameters $M=N=400$. The method is capable to approximate accurately the dynamics that we have obtained in \cite{Urokinase_paper} including the emergence and movement of multiple steep peeks. The present simulation clearly demonstrates the robustness of the newly developed scheme to simulate complex taxis-diffusion systems arising in cell biology.

To investigate the dynamics of such a cancer invasion system in the case that the cell migration is restricted by the occupied extracellular space we have endowed the model \eqref{eq:uPA} with the volume filling approach \eqref{eq:vol_filling}. In more details we have replaced the evolution equation for the tumor cell density in \eqref{eq:uPA} by
\begin{equation}\label{eq:uPA_vol_filling}
	\ddt \rho = \ddx \( D_\rho (1 + (\gamma -1)\rho^\gamma)\ddx \rho -  (1-\rho^\gamma)(\chi_u \rho\ddx u + \chi_p \rho \ddx p + \chi_v \rho\ddx v)\) +\mu_1\rho(1-\rho)
\end{equation}
and resolved the same numerical experiment as above. To this end the scheme has been adapted in a similar way as in \eqref{eq:vol_filling_scheme}.
\begin{figure}[!t]
	\centering
	\includegraphics[trim=0 0.5 0 1cm, clip]{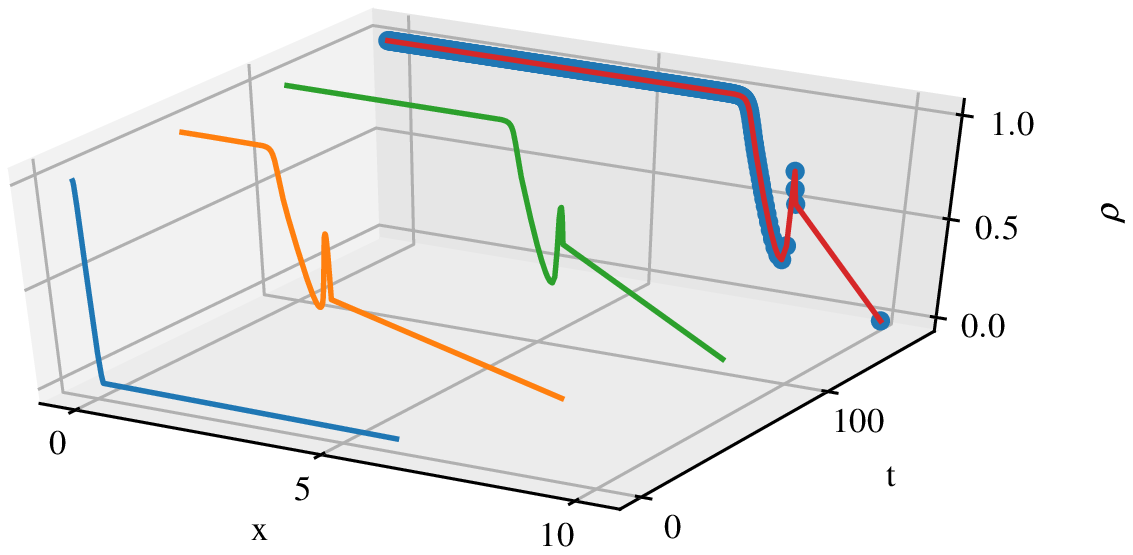}
	\includegraphics[scale=0.98, trim=0 0 0 1.2cm, clip]{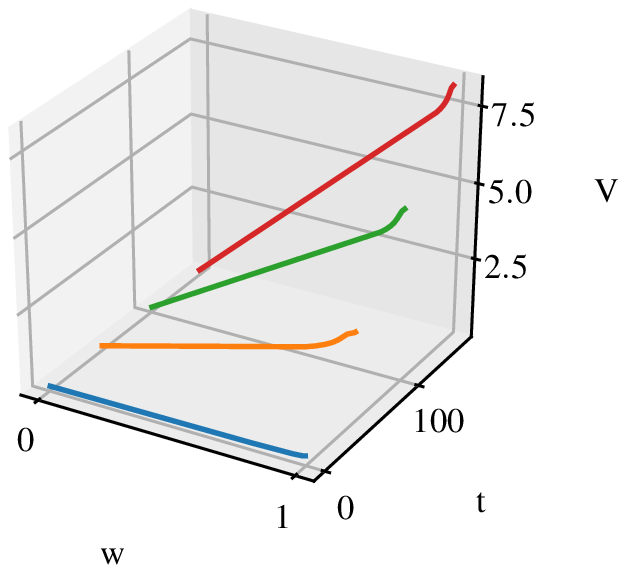}
	\includegraphics[scale=0.98, trim=0 0 0 1.2cm, clip]{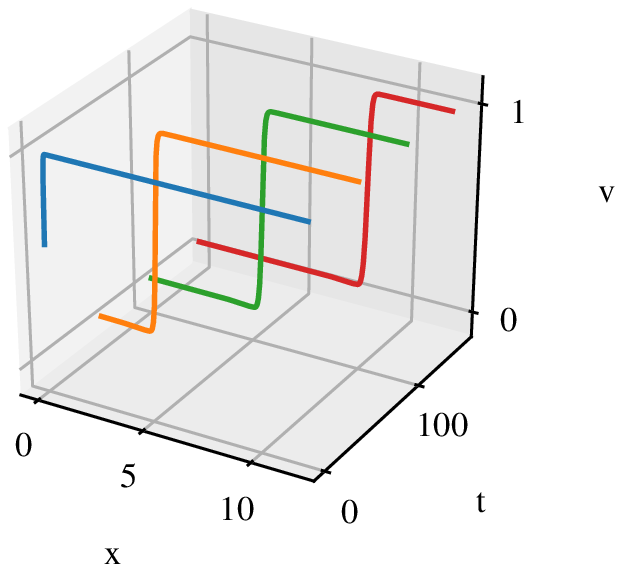}
	\includegraphics[scale=0.98, trim=0 0.2cm 0 0.5cm, clip]{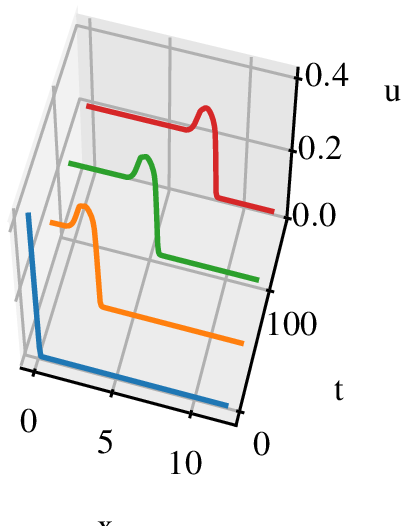}
	\includegraphics[scale=0.98, trim=0 0.2cm 0 0.5cm, clip]{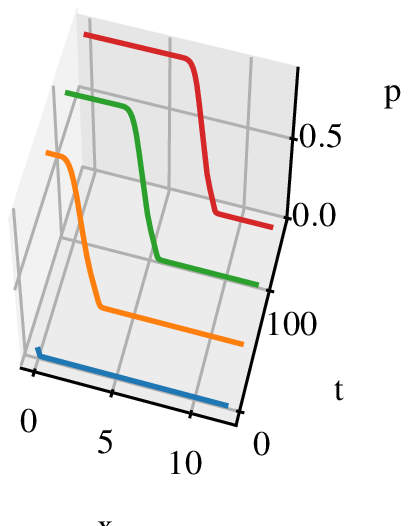}
	\includegraphics[scale=0.98, trim=0 0.2cm 0 0.5cm, clip]{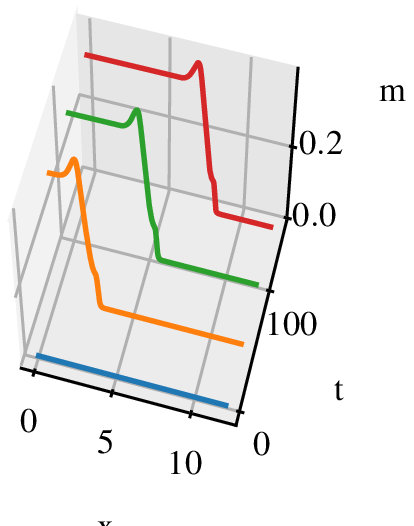}
	\caption{Numerical results (cancer cell concentration, inverse cumulative function, ECM, uPA, PAI-1 and plasmin density in space and time) in the model \eqref{eq:uPA} with volume filling by \eqref{eq:uPA_vol_filling} and exponent $\gamma=2$. We have used $M=N=400$ grid points on both meshes in the numerical simulation.}\label{fig:chaplol_vf2}
\end{figure}
\begin{figure}[!t]
	\centering
	\includegraphics[trim=0 0.5 0 1cm, clip]{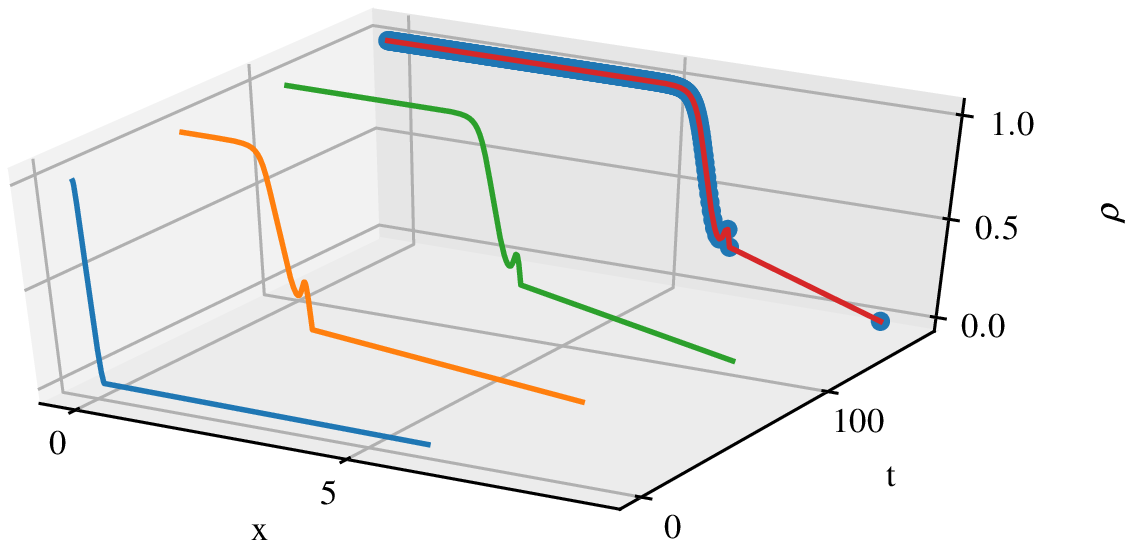}
	\includegraphics[scale=0.98, trim=0 0 0 1.2cm, clip]{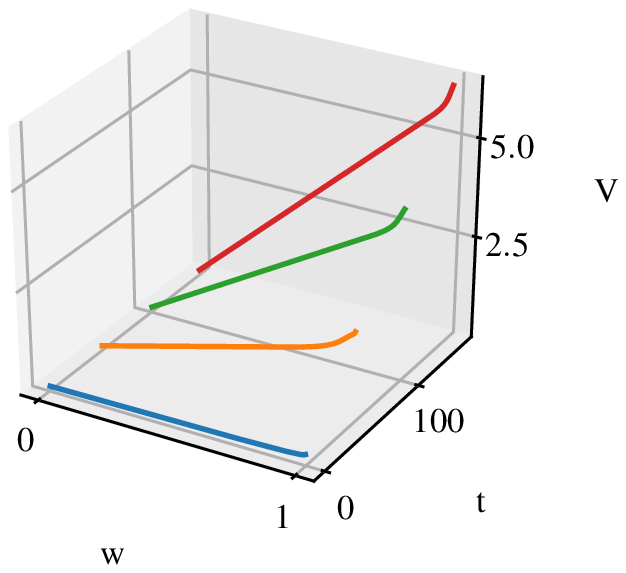}
	\includegraphics[scale=0.98, trim=0 0 0 1.2cm, clip]{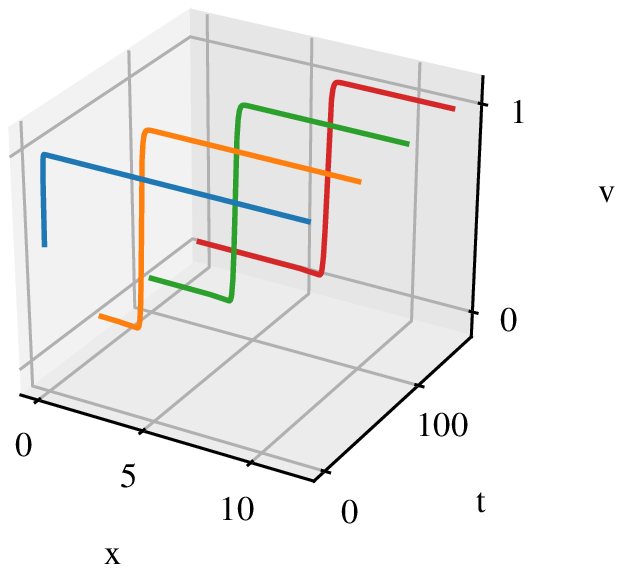}
	\includegraphics[scale=0.98, trim=0 0.2cm 0 0.5cm, clip]{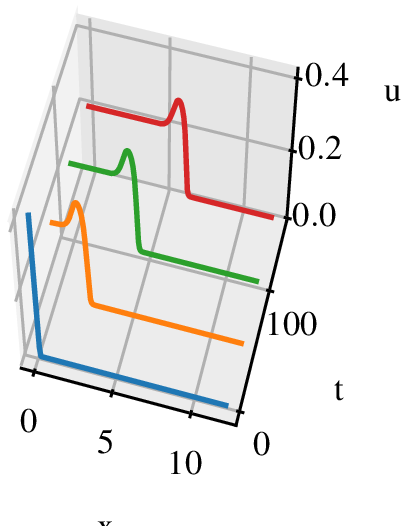}
	\includegraphics[scale=0.98, trim=0 0.2cm 0 0.5cm, clip]{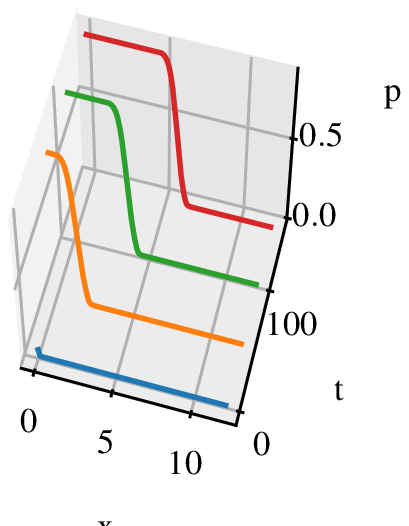}
	\includegraphics[scale=0.98, trim=0 0.2cm 0 0.5cm, clip]{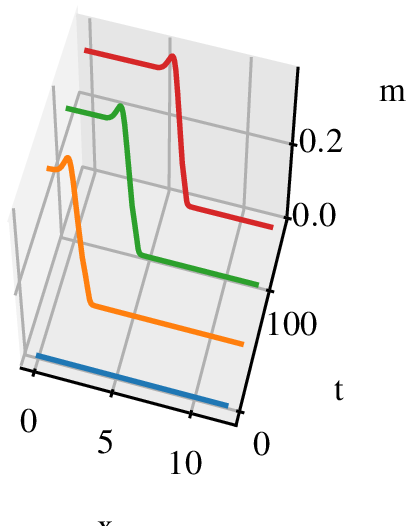}
	\caption{Numerical results (cancer cell concentration, inverse cumulative function, ECM, uPA, PAI-1 and plasmin density in space and time) in the model \eqref{eq:uPA} with volume filling by \eqref{eq:uPA_vol_filling} and exponent $\gamma=0.5$. We have used $M=N=400$ grid points on both meshes in the numerical simulation.}\label{fig:chaplol_vf0_5}
\end{figure}

In Figures \ref{fig:chaplol_vf2} and \ref{fig:chaplol_vf0_5} we show the simulation results for the exponents chosen $\gamma=2$ and $\gamma=0.5$, where we have used $M=N=400$ mesh points in the computation. Contrary to the simulations without volume filling, the cancer cells do not exhibit the rich dynamics, i.e. the formation of multiple clusters. Instead a single concentration of tumor cells invades the ECM and leaves a homogeneous distribution of tumor cells of maximal density $\rho=1$ behind. Reducing the diffusivity of the cells by decreasing the exponent $\gamma$ results in a slower invasion of the tissue and to a lower concentration at the invading front of tumor cells. This can be seen when comparing Figure \ref{fig:chaplol_vf2} ($\gamma=2$) and Figure \ref{fig:chaplol_vf0_5} ($\gamma=0.5$).

To study how the new method compares in efficiency to more conventional numerical methods we consider again the above experiment without volume filling. For the comparison we consider the Finite Volume/Finite Difference from \cite{Urokinase_paper} for both uniform and adaptive meshes. In particular we have chosen a second order method with implicit-explicit Strang operator splitting. For the adaptive mesh refinement (AMR) method we have chosen the gradient monitor function to determine the mesh-cells to be either refined or coarsened\footnote{In more details we have used the refinement and coarsening threshold values $\theta_\text{ref} = 10,~\theta_\text{coars} = 2.5$, a single refinement and coarsening operation per time step $n_\text{ref}=n_\text{coars}=1$ and a maximal refinement level of $l_\text{max}=2$, cf \cite{Urokinase_paper}.}. For brevity, we will refer to the adaptive method as AMR and to the uniform method as FVFD. The new mass-transport/finite element method will be denoted by MTFE.

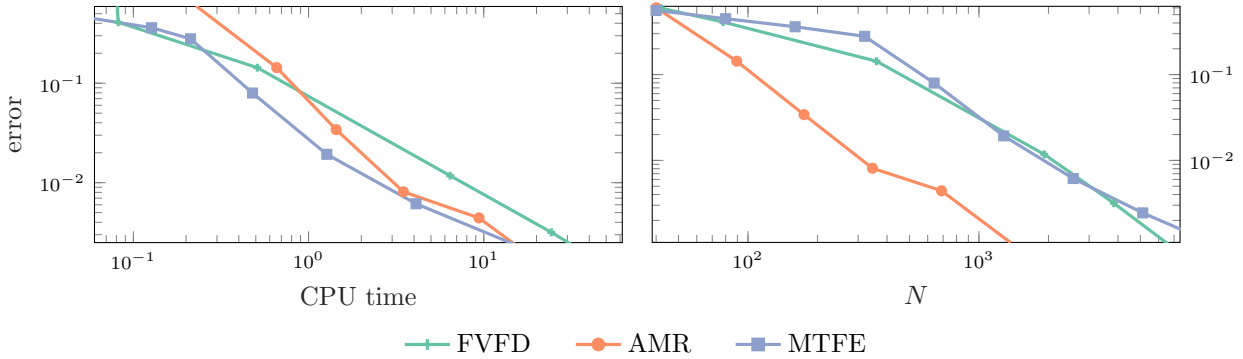
\begin{figure}
	\centering
	\setlength\figureheight{0.14\textheight}
	\setlength\figurewidth{.85\textwidth}
	\input{mtfe_vs_amr_vs_fvfd.tex}
	\caption{Relation between the CPU time and the error (left) and between the (average) number of cells and the error (right) for the FVFD, AMR, and MTFE scheme in log-log scale in a numerical experiment with the uPA model \eqref{eq:uPA}. The new MTFE method seems to be most efficient in terms of error per CPU time, its relation between the error and the average number of cells is similar as in the FVFD scheme.}
	\label{fig:mtfe_vs_amr}
\end{figure}

For our comparism we consider the set $S= \{ 40, 80, 160, 320, 640, 1280\}$ and run the MTFE method for $M \in S$, the FVFD method for $N=6k$ for any $k \in S$, and the AMR method for $N_0 \in S$ with $N_0$ denoting the number of cells on the lowest level. We couple the two meshes in the MTFE scheme by setting $N= M$. We do not consider finer resolutions due to restrictions by the uniform reference solution in the error computations of solutions obtained by the MTFE scheme. For comparison reasons we let $N$ denote the average number of cells in the AMR method. In addition, all three methods employ the same Courant number $CFT=0.49$ and all numerical solutions are computed on the domain $\Omega = (0,5)$.

We compute the numerical solutions of the considered experiment at the time instance $t=23$ that features two steep peaks in the cancer cell concentration. In this process we measure the CPU time that is needed for the corresponding simulations and compute the error of the approximation at the final time. For the error computation we have used a reference solution that employs a uniform mesh with cell size $h=1.25 \times 10^{-5}$ in the relevant part of the domain\footnote{We have computed a uniform solution in $(0,2)$ with our uniform method using $N=160\,000$ mesh cells.}. The discrete $L^1$ error is then computed with respect to the densities using a suitable projection of the reference solution. 
Note that the following test results are dependent on our (non-reference) implementation of the numerical methods.

We show the results of our comparison in Figure \ref{fig:mtfe_vs_amr}. Here we present the relation between the error and the computation time and the relation between the error and the average number of cells for all three methods. We see that for all tested methods the error decreases as either the cell number or the CPU time increases.
Figure \ref{fig:mtfe_vs_amr} (left) exhibits an advantage of the new MTFE method over the other schemes in efficiency for most of the conducted simulations. This can be seen as the MTFE method achieves in most cases lower errors than the FVFD or the AMR scheme using the same CPU time.
As the runtime increases the MTFE method approaches the efficiency of the AMR method with the new method being at a slight advantage over the mesh refinement method. Clearly, the AMR and the MTFE scheme both outperform the FVFD method for sufficiently large CPU times.

Figure \ref{fig:mtfe_vs_amr} (right) shows that the AMR method achieves the lowest errors when compared with simulations by the FVFD and MTFE scheme employing the same average number of cells. The error of the MTFE scheme has a similar dependence on the number of cells as the error of the FVFD scheme. We conjecture thus that the better efficiency of the MTFE scheme in terms of CPU time seen in Figure \ref{fig:mtfe_vs_amr} (left) is probably caused by the CFL condition in the MTFE scheme allowing for larger time steps compared to the FVFD method.

\section{Conclusion}
In this paper we have proposed a new splitting scheme for one-dimensional reaction-taxis-diffusion systems related to the Keller-Segel system.  The solutions of these systems are well known for having concentrated and diffusive regions simultaneously. In addition, traveling waves and merging phenomena typically occur.

Our splitting has separated a part of the model which is mass conservative in the cell density from the rest of the system. The latter has been approximated by a classical linear finite element method, whereas the approximation of the conservative part has been based on the mass transport strategy. More precisely, we have first transformed the cell density to the corresponding pseudo-inverse cumulative distribution. Then we have discretized the transformed system by the finite difference method and used a cubic spline to account for the chemo-attractant whose evolution is described in the rest subsystem. The splitting method is described in 
Section~\ref{section:num}. In Lemma~\ref{lem:cfl} we have studied the stability of the explicit mass transport method for the conservative part in which we allowed for general nonlinear diffusion. The obtained result has been used to derive a time-step restriction for our scheme.

In Section \ref{sec:experiments} we have presented a series of numerical experiments demonstrating the robustness and reliability of the scheme. In particular, we have used the new method to resolve the Keller-Segel model in the parabolic-elliptic and in the parabolic-parabolic form numerically. We have applied our scheme also to augmentations of these systems by reaction terms, nonlinear diffusion and a volume filling approach. The method has resolved the movement, splitting and aggregation phenomena accurately.  We have verified the mesh convergence of the scheme in both time and space in an application to a simple tumor invasion system in Section \ref{sec:cancer1}. The obtained experimental order of convergence has ranged around two spatially and temporally. Moreover, we have applied the scheme to the uPA-tumor invasion model from \cite{Chaplain.2005} in Section \ref{sec:uPA}. The proposed hybrid
mass transport finite element scheme has been capable to resolve its complex dynamics featuring multiple peaks in the cancer cell concentration without using a fine spatial discretization. By the help of our new method we could also study a combination of the uPA model with the volume filling approach from \cite{HPVolume}.
In addition, we have compared the efficiency of the hybrid mass transport finite element method with a finite volume scheme with adaptive mesh refinement from \cite{Urokinase_paper}. The hybrid mass transport finite element method has not only outperformed the uniform finite volume scheme but it has also delivered slightly better results than the finite volume scheme equipped with adaptive mesh refinement.

%%%%%%%%%%%%%%%%%%%%%%%%%%%%%%%%%%%%%%

\section*{Acknowledgments}
JAC was partially supported by the Royal Society via a Wolfson Research Merit Award and by EPSRC grant number EP/P031587/1. NK was supported by the Max-Planck Graduate Center of the University Mainz.
The research of ML was partially supported by the German Science Foundation (DFG) under the grant TRR 146 \enquote{Multiscale simulation methods for soft matter systems}.
%%%%%%%%%%%%%%%%%%%%%%%%%%%%%%%%%%%%%%

\bibliographystyle{plain} %{plain}
	\bibliography{ckl}
\end{document}

%% file: mtfe_vs_amr_vs_fvfd.tex
%\tikzsetnextfilename{mtfe_vs_amr_fig}
\definecolor{mycolor1}{rgb}{0.85000,0.33000,0.10000}%
\definecolor{mycolor2}{rgb}{0.00000,0.45000,0.74000}%
\begin{tikzpicture}
\begin{groupplot}[
/tikz/mark size=1.5pt,
group style={
	group name=my plots,
	group size=2 by 1,
	horizontal sep=4mm,      % <-- default: 1cm
	%vertical sep=0.5cm,        % <-- default: 1cm
},
scale only axis,
xmode=log,
xminorticks=true,
xlabel style={font=\color{white!15!black}},
ymode=log,
width= .5 \figurewidth,
height= \figureheight,
ticklabel style = {font=\scriptsize},
yminorticks=true,
ylabel style={font=\color{white!15!black}},
axis background/.style={fill=white},
]
\nextgroupplot[%
xmin=0.06,
xmax=61.76,
xlabel={CPU time},
ymin=0.0025,
ymax=0.59,
ylabel={error},
legend to name = {complegend}, legend columns=3, legend style={/tikz/every even column/.append style={column sep=0.5cm}, draw=none}
]
\coordinate (c2) at (rel axis cs:1,1);
\addplot [index of colormap=0 of Set2, mark=+, line width=1.2pt]
table[row sep=crcr]{%
	0.0810149999999999	0.595096740090057\\
	0.081945	0.409912313794213\\
	0.50991	0.143247247126005\\
	6.428431	0.0117598980920974\\
	24.302709	0.00317187068113558\\
	54.9617194974483	0.00133280631028635\\
};
\addlegendentry{FVFD}

\addplot [index of colormap=1 of Set2, mark=*, line width=1.2pt]
table[row sep=crcr]{%
	0.227127	0.600697757360212\\
	0.658041	0.143407533858166\\
	1.436811	0.0341480271689113\\
	3.479309	0.00810267891942928\\
	9.388393	0.00442482418705824\\
	27.508583	0.00109799947823093\\
};
\addlegendentry{AMR}

\addplot [index of colormap=2 of Set2, mark=square*, line width=1.2pt]
  table[row sep=crcr]{%
	0.0236990451812744	0.559003464077995\\
	0.0581839084625244	0.44853613573295\\
	0.126996040344238	0.361207432410256\\
	0.212043046951294	0.27939446296268\\
	0.478769063949585	0.0796779900369774\\
	1.27087187767029	0.0192975331674551\\
	4.10677099227905	0.00614724744937904\\
	14.5441420078278	0.00244535301999092\\
	103.272348165512	0.00108870607392653\\
};
\addlegendentry{MTFE}
\nextgroupplot[%
xmin=38.4816297054039,
xmax=7429.63950759495,
xminorticks=true,
xlabel={$N$},
ymin=0.00110594054081847,
ymax=0.625055192527397,
yticklabel pos=right,
yminorticks=true,
cycle list name=Set2,
legend style={legend cell align=left, align=left, draw=white!15!black}
]
\coordinate (c1) at (rel axis cs:0,1);
\addplot[index of colormap=0 of Set2, mark=+, line width=1.2pt]
table[row sep=crcr]{%
	42	0.595096740090057\\
	78.0000000000001	0.409912313794213\\
	360	0.143247247126005\\
	1919.99999999999	0.0117598980920974\\
	3839.99999999998	0.00317187068113558\\
	7034.52897856029	0.000928063984909919\\
};
\addplot[index of colormap=1 of Set2, mark=*, line width=1.2pt]
table[row sep=crcr]{%
	40	0.600697757360212\\
	89.3744637312648	0.143407533858166\\
	174.773119544118	0.0341480271689113\\
	345.931881447674	0.00810267891942928\\
	688.24129331835	0.00442482418705824\\
	1373.00370358649	0.00109799947823093\\
};
\addplot[index of colormap=2 of Set2, mark=square*, line width=1.2pt]
  table[row sep=crcr]{%
	40	0.559003464077995\\
	80	0.44853613573295\\
	160	0.361207432410255\\
	320	0.27939446296268\\
	640	0.0796779900369773\\
	1280	0.0192975331674551\\
	2560	0.00614724744937904\\
	5120.00000000001	0.00244535301999092\\
	10240	0.00108870607392653\\
};
\end{groupplot}
\coordinate (c3) at ($(c1)!.5!(c2)$);
\node[below] at (c3 |- current bounding box.south){\pgfplotslegendfromname{complegend}};
\end{tikzpicture}%